\documentclass[12pt]{article}

\usepackage{multirow} %% for nulti rows in the table
\usepackage{graphicx} %% need for scale table

\usepackage{amsmath}
\usepackage{amsthm} %% package amsthm is needed for the enviorment proof
\usepackage{hyperref}
\usepackage{amsmath}
\usepackage{amssymb}
\usepackage{fullpage}
\usepackage{placeins}
\usepackage{color}

\newtheorem{thm}{Theorem}[section]
\newtheorem{lem}{Lemma}[section]

\newtheorem{rmk}{Remark}[section]
\newtheorem{cor}{Corollary}[section]

\renewcommand{\th}{\theta}
\newcommand{\be}{\begin{equation}}
\newcommand{\ee}{\end{equation}}

\newcommand{\ur}{u_{r}}
\newcommand{\uthe}{u_{\theta}}

\newcommand{\dive}{\mbox{div}}

\newcommand{\pt}{\partial}

\newcommand{\ctthe}{\cot\theta}
\newcommand{\sthe}{\sin\theta}
\newcommand{\cthe}{\cos\theta}

\newcommand{\er}{e_r}
\newcommand{\ethe}{e_{\theta}}
\newcommand{\ephi}{e_{\phi}}

\renewcommand{\th}{\theta}

\newcommand{\bee}{\begin{equation*}}
\newcommand{\eee}{\end{equation*}}

\begin{document}
\title{Homogeneous solutions of stationary Navier-Stokes equations with isolated singularities on the unit sphere. II. Classification of axisymmetric no-swirl solutions}
\author{Li Li\footnote{Department of Mathematics, Harbin Institute of Technology, Harbin 150080, China. Email: lilihit@126.com}, 
YanYan Li\footnote{Department of Mathematics, Rutgers University, 110 Frelinghuysen Road, Piscataway, NJ 08854, USA. Email: yyli@math.rutgers.edu}, 
Xukai Yan\footnote{Department of Mathematics, Rutgers University, 110 Frelinghuysen Road, Piscataway, NJ 08854, USA. Email: xkyan@math.rutgers.edu}}
\date{}
\maketitle
\abstract{We classify all (-1)-homogeneous axisymmetric no-swirl solutions of incompressible stationary Navier-Stokes equations in three dimension which are smooth  on the unit sphere minus the south and north poles, parameterizing them as a four dimensional surface with boundary in appropriate function spaces. Then we establish smoothness properties of the solution surface in the four parameters. The smoothness properties will be used in a subsequent paper where we study the existence of (-1)-homogeneous axisymmetric solutions with non-zero swirl on $\mathbb{S}^2\setminus\{S,N\}$, emanating from the  four dimensional solution surface. }

\setcounter{section}{0}

\section{Introduction}
Consider the incompressible stationary Navier-Stokes equations (NSE) in $\mathbb{R}^3$: 
\begin{equation}\label{NS}
\left\{
\begin{split}
	& -\triangle u + u\cdot \nabla u +\nabla p = 0, \\
	& \dive\textrm{ } u=0.
\end{split}
\right. 
\end{equation} 

The equations are invariant under the scaling $u(x)\to \lambda u(\lambda x)$ and $p(x)\to \lambda^2 p(\lambda x)$, $\lambda>0$. We study solutions which are invariant under the scaling. For such solutions $u$ is (-1)-homogeneous and $p$ is (-2)-homogeneous.  We call them (-1)-homogeneous solutions according to the homogeneity of $u$.

We will write the NSE (\ref{NS}) in spherical coordinates $(r,\theta,\phi)$. A vector field $u$ can be written as
\[
%\label{u_polar}
	u = u_r \er + u_\theta \ethe + u_\phi \ephi, 
\]
where 
\[
	\er = \left(
	\begin{matrix}
		\sthe\cos\phi \\
		\sthe\sin\phi \\
		\cthe
	\end{matrix} \right),  \hspace{0.5cm}
	\ethe = \left(
	\begin{matrix}
		\cthe\cos\phi  \\
		\cthe\sin\phi   \\
		-\sthe	
	\end{matrix} \right), \hspace{0.5cm}
	\ephi = \left(
	\begin{matrix}
		-\sin\phi \\  \cos\phi \\ 0
	\end{matrix} \right).  
\]
A vector field $u$ is called axisymmetric if $u_r$, $u_{\theta}$ and $u_{\phi}$ are independent of $\phi$, and is called {\it no-swirl} if $u_{\phi}=0$.  

Landau discovered in \cite{Landau} a three parameter family of explicit (-1)-homogeneous  solutions of the stationary NSE (\ref{NS}), which are axisymmetric and with no swirl. These solutions are now called Landau solutions. The NSE (1) in the axisymmetric no-swirl case was converted earlier to an equation of Riccati type by Slezkin in [11].   The Riccati type equation was later independently derived
by Yatseyev using a  different method in \cite{Y1950}, where various exact solutions were given.
The Landau solutions were also independently found by Squire in \cite{Squire}. Tian and Xin proved in \cite{TianXin} that all (-1)-homogeneous, axisymmetric nonzero solutions of (\ref{NS}) in $C^{2}(\mathbb{R}^3\setminus\{0\})$ are Landau solutions. 
A classification of all (-1)-homogeneous solutions was given by \v{S}ver\'{a}k in \cite{Sverak}: all (-1)-homogeneous nonzero solutions of (\ref{NS}) in $C^2(\mathbb{R}^3\setminus\{0\})$ are Landau solutions. He also proved in the same paper that there is no nonzero (-1)-homogeneous 
solution of the stationary NSE in $C^{2}(\mathbb{R}^n\setminus\{0\})$ for $n\ge 4$.
 In dimension $n=2$, he characterized all such solutions satisfying a zero flux condition.
%\textbf{Theorem A (\cite{Sverak})} \emph{All (-1)-homogeneous nonzero solutions of (\ref{NS}) in $C^2(\mathbb{R}^3\setminus\{0\})$ are Landau solutions.}\\

In \cite{Serrin}, Serrin modeled the tornado by (-1)-homogeneous axisymmetric solutions of the three dimensional incompressible stationary Navier-Stokes equations in the half space with zero boundary conditions and one singularity on the unit sphere.

%Some recent works also study homogeneous solutions of the stationary NSE.   
More recently,  Karch and Pilarczyk showed in \cite{Karch} that Landau solutions are asymptotically stable under any $L^2$ perturbations. %As an extension, we would like to study the stability of the (-1)-homogeneous axisymmetric solutions which are smooth on $\mathbb{S}^2\setminus\{S,N\}$. 
Classifications of homogeneous solutions to the $2$-dimensional and $3$-dimensional stationary Euler equations are studied respectively in \cite{Shvydkoy} by Luo and Shvydkoy, and in \cite{Shvydkoy3d} by Shvydkoy.
More studies on (-1)-homogeneous axisymmetric solutions of the stationary NSE (\ref{NS}) can be found in  \cite{Goldshtik}, \cite{PAULL1}, \cite{PAULL2}, \cite{PAULL3}, \cite{Serrin} and \cite{WangSteady}. 

We are interested in analyzing solutions which are smooth on $\mathbb{S}^2$ minus finite points.
 We have classified in \cite{LLY1} all axisymmetric no-swirl solutions with one singularity at the south pole. They form a two dimensional surface with boundary in appropriate function spaces.  These solutions are among the solutions found in \cite{Y1950}, where the solutions were obtained by a different method. It was proved in \cite{LLY1} that there are no other solutions with precisely one singularity at the south pole.   It was also proved there that there exists  a curve of axisymmetric solutions with nonzero swirl  emanating from every point in the interior and one part of the boundary of the surface of no-swirl solutions, while there is no such curve from any point on the other part of the boundary. Uniqueness  results of nonzero swirl solutions near  the no-swirl solution surface were also given in \cite{LLY1}.  Our main result in this paper is the classification of  all (-1)-homogeneous, axisymmetric no-swirl solutions of (\ref{NS}) which are smooth on $\mathbb{S}^2\setminus\{S,N\}$, where $S$ is the south pole and $N$ is the north pole. They are identified as a 4-dimensional surface with boundary in appropriate function spaces. We have established smoothness properties   of the solutions surface in the four parameters. These properties are used in a subsequent paper \cite{LLY3} where we study the existence of (-1)-homogeneous axisymmetric solutions with non-zero swirl on $\mathbb{S}^2\setminus\{S,N\}$, emanating from the 4-dimensional solution surface.

A (-1)-homogeneous axisymmetric  vector field $u$ is divergence free if and only if 
\begin{equation}\label{eq_divefree}
    \ur =-\frac{d \uthe}{d \theta} -  \uthe \ctthe.
    \end{equation}

     We work with a new unknown function and a different independent variable: 
\begin{equation}\label{Va}
	x:=\cthe, \quad U_\theta := u_\theta \sthe. 
\end{equation}

     As explained in \cite{LLY1}, $(u,p)$ is a (-1)-homogeneous axisymmetric no-swirl solution of (\ref{NS}) if and only if %where $x=\cos\theta$, $U_{\theta}=\sin\theta u_{\theta}$, 
$u_{\phi}=0$,  $u_r$ is given by (\ref{eq_divefree}), $p$ is given by 
\[
      p=-\frac{1}{2}\left(\frac{d^2 \ur}{d\theta^2} + (\ctthe - \uthe) \frac{d \ur}{d\theta} + \ur^2 +\uthe^2\right),
\]
and $U_{\theta}$ satisfies, for  some constants $c_1,c_2,c_3\in \mathbb{R}$, 
 \be\label{eq:UthP}
	(1-x^2)U_\th' + 2x U_\th + \frac{1}{2} U_\th^2 = P_c(x):=c_1(1-x)+ c_2 (1+x) + c_3(1-x^2),
\ee
 where " $ ' $ " denotes differentiation in $x$, and $c:=(c_1,c_2,c_3)$. 

%!TEX encoding = UTF-8 Unicode
For each $c_1\ge -1$ and  $c_2\ge -1$, define
\be\label{sec2:eq:bar:c3}
	\bar{c}_3 (c_1,c_2) := -\frac{1}{2} \left( \sqrt{1+c_1} + \sqrt{1+c_2}  \right) \left( \sqrt{1+c_1} + \sqrt{1+c_2}  + 2 \right). 
\ee

Define 
\begin{equation*}%\label{eqJ}
   J:=\{c\in \mathbb{R}^3 \mid c_1\geq -1, c_2\geq -1,c_3\ge \bar{c}_3(c_1,c_2)\}. 
\end{equation*}

\begin{thm}\label{thm1_1}
There exist $U^{\pm}_{\theta}(c)(x)\in C^0(J\times [-1,1])$, such that for every $c\in J$,  $U^{\pm}_{\theta}(c)\in C^{\infty}(-1,1)$ satisfy (\ref{eq:UthP}) in $(-1,1)$, and $U^{-}_{\theta}(c)\le U_{\theta}\le U^+_{\theta}(c)$ for any solution $U_{\theta}$ of (\ref{eq:UthP}) in $(-1,1)$. Moreover, if $c_3>\bar{c}_3(c_1,c_2)$,  $U^{-}_{\theta}<U^{+}_{\theta}$in $(-1,1)$,  and   if $c_3=\bar{c}_3(c_1,c_2)$,
\begin{equation}\label{eq1_1_1}
   U_{\theta}^+(c)=U^-_{\theta}(c)=U^*_{\theta}(c_1,c_2):= (1+\sqrt{1+c_1})(1-x)+(-1-\sqrt{1+c_2})(1+x).
\end{equation}
\end{thm}

%The solution  $U^*_{\theta}(c_1,c_2)$  is called the conically similar flow generated by a half-line volume source in literature, see \cite{PAULL1}.
\bigskip

Next, for $c\in J$, introduce 
\begin{equation*}
    \gamma^+(c):=U^+_{\theta}(c)(0), \quad \gamma^-(c):=U^-_{\theta}(c)(0).
\end{equation*}
Define
\begin{equation*}
I:= \{(c, \gamma)\in \mathbb{R}^4 \mid c_1\geq -1, c_2\geq -1, c_3\geq \bar{c_3}(c_1,c_2). \gamma^-(c)\leq \gamma \leq \gamma^+(c) \}.
\end{equation*}

\begin{thm}\label{thm1_2}
	 %The functions $\gamma^-,\gamma^+\in  C^0(J)$. 
	  For each $(c,\gamma)$ in $I$,  
	 equation (\ref{eq:UthP}) has a unique solution $U^{c,\gamma}_{\theta}$ in $C^{\infty}(-1,1)\cap C^0[-1,1]$ satisfying $U^{c,\gamma}_{\theta}(0)=\gamma$.
Moreover, these are all (-1)-homogeneous axisymmetric no-swirl solutions of the Navier-Stokes equations (\ref{NS}) on $\mathbb{S}^2\setminus\{S,N\}$.  
\end{thm}

Clearly, $U^{c,\gamma^{\pm}(c)}_{\theta}=U^{\pm}_{\theta}(c)$ for $c\in J$. Theorem \ref{thm1_1} and Theorem \ref{thm1_2} give a classification of all (-1)-homogeneous axisymmetric, no-swirl solutions of Navier-Stokes equations in $C^2(\mathbb{S}^2\setminus \{S, N\})$. 
There is a 1-1 correspondence between $U_\theta^{c,\gamma}$ and points in the four dimensional surface $I$. 

 Recall that Landau solutions are 
 \[
    U_{\theta}(x)=\frac{2(1-x^2)}{x+\lambda},\quad |\lambda|>1,
 \]
 and they correspond to $U_{\theta}^{c,\gamma}$ with $c=0$ and $\gamma\in(-2,2)\setminus\{0\}$.
	%Landau solutions correspond to $U^{c,\gamma}$ where $(c,\gamma)\in I$, $c=0, \gamma\in(-2,2)$ and $\gamma\not=0$. 
	
	The solutions in $C^{\infty}(\mathbb{S}^2\setminus\{S\})$ correspond to $U_{\theta}^{c,\gamma}$ with $c_2=0$, $c_1=-2c_3$ and $\gamma^-(c)<\gamma\le \gamma^+(c)$.

\bigskip

Define 
\be\label{eq:tau12}
	\tau_1(c_1) := 2-2\sqrt{1+c_1}, \quad \tau_2(c_1) := 2+2\sqrt{1+c_1}, 
\ee
\be\label{eq:tau12'}
	\tau_1'(c_2) := -2-2\sqrt{1+c_2}, \quad \tau_2'(c_2) := -2+2\sqrt{1+c_2}. 
\ee

\begin{thm}\label{thm1_3}
Suppose $(c,\gamma)\in I$, then

(i) If $c_3>\bar{c}_3 (c_1,c_2)$, then $\gamma^-(c)<\gamma^+(c)$, and for any $\gamma^-(c) \leq \gamma < \gamma' \leq \gamma^+(c)$, $U_\theta^{c,\gamma} < U_\theta^{c,\gamma'}$ in $(-1,1)$. Moreover,		
\[
  \begin{split}
	   & \hspace{0.8cm} \{(x,y)\mid -1<x<1, U_{\theta}^{c,\gamma^-(c)} (x)\leq y \leq U_\theta^{c,\gamma^+(c)}(x) \}\\
	   & = \bigcup_{\gamma\in [\gamma^-(c),\gamma^+(c)]} \{(x,U_{\theta}^{c,\gamma}(x)) \mid -1<x<1\}.
	     \end{split}
	\]

	(ii) %For any $(c,\gamma)\in I$, $U_\theta^{c,\gamma}(\pm 1)$ both exist and are finite. Moreover, %if $\gamma^-(c)<\gamma<\gamma^+(c)$, then
	\begin{equation*}%\label{sec2:eq:tau}
		U_\theta^{c,\gamma}(-1) := \left\{
		\begin{array}{ll}
			\tau_2(c_1), &  \mbox{when } \gamma=\gamma^+(c), \\
			\tau_1(c_1), &  \mbox{otherwise},
		\end{array}
		\right. \quad 
		U_\theta^{c,\gamma}(1) := \left\{
		\begin{array}{ll}
			\tau_1'(c_2), &  \mbox{when } \gamma=\gamma^-(c), \\
			\tau_2'(c_2), &  \mbox{otherwise}. 
		\end{array}
		\right. 
	\end{equation*}
\end{thm}

In addition to the continuity of $\gamma^+(c)$ and $\gamma^-(c)$ in $J$, they have further smoothness properties.

\begin{thm}\label{thm:c1c2}
$\gamma^+(c)$ is  in $C^{\infty}(J\setminus\{c\mid c_1=-1\})$, and $\gamma^+(-1,c_2,c_3)$ is  in $C^{\infty}(J\cap\{c\mid c_1=-1\})$  as a function of $(c_2,c_3)$. $\gamma^-(c)$ is  in $C^{\infty}(J\setminus\{c\mid c_2=-1\})$, and $\gamma^-(c_1,-1,c_3)$ is  in $C^{\infty}(J\cap\{c \mid c_2=-1\})$  as a function of $(c_1,c_3)$.
\end{thm}

\bigskip

\bigskip

We also have the smoothness properties of $U^{c,\gamma}_{\theta}$ in $(c,\gamma)$.  Let the subsets $J_k$, $1\le k\le 4$, of $J$ be defined as 
\begin{equation*}%\label{eq5_1_J_1}
   \begin{split}
      & J_1:=\{c\in J \mid c_1>-1,c_2>-1,c_3>\bar{c}_3\},\quad J_2:=\{c\in J \mid c_1=-1,c_2>-1,c_3>\bar{c}_3\},\\
      & J_3:=\{c\in J \mid c_1>-1,c_2=-1,c_3>\bar{c}_3\},\quad J_4:=\{c\in J \mid c_1=-1,c_2=-1,c_3>\bar{c}_3\}.
   %   & J_5:=\{c\in J |c_1>-1,c_2>-1,c_3=\bar{c}_3\},\quad J_6:=\{c\in J |c_1=-1,c_2>-1,c_3=\bar{c}_3\},\\
   %   & J_7:=\{c\in J |c_1>-1,c_2=-1,c_3=\bar{c}_3\},\quad J_8:=\{c\in J |c_1=-1,c_2=-1,c_3=\bar{c}_3\}.\\
   \end{split}
\end{equation*}

We define the following subsets of $I$:
for $1\le k\le 4$, let
\[
\begin{split}
    & I_{k,1}:=\{(c,\gamma)\in I \mid c\in J_k,\gamma^-(c)<\gamma<\gamma^+(c)\}, \\
    & I_{k,2}:=\{(c,\gamma)\in I \mid c\in J_k,\gamma=\gamma^+(c)\},\\
    & I_{k,3}:=\{(c,\gamma)\in I \mid c\in J_k,\gamma=\gamma^-(c)\}.
    \end{split}
\]
%and for $5\le k\le 8$, let $I_{k,l}:=J_k$, $l=1,2,3$.

%The solution $U^{c,\gamma}_{\theta}$ given by Theorem \ref{thm1_2} has the following property:

As mentioned earlier, the following estimates of $U^{c,\gamma}_{\theta}$ are needed in our next paper on the existence of (-1)-homogeneous axisymmetric solutions of (\ref{NS}) with nonzero swirl on $\mathbb{S}^2\setminus\{S,N\}$.

\begin{thm}\label{propA_1}
Let $K$ be a compact set contained in one of $I_{k,l}$, $1\le k\le 4$, $l=1,2,3$. Then $U^{c,\gamma}_{\theta}$ is in  $C^{\infty}(K\times(-1,1))$. Moreover,	

	(i) If  $k=1$ and $l=1,2,3$, or $(k,l)=(2,2)$ or $(3,3)$, then for $-1<x<1$, 
	\begin{equation}\label{eqA_P_1}
		|\pt_c^{\alpha} \pt_\gamma^j U^{c,\gamma}_\theta(x)|\le C(m,K), \quad \mbox{for any } 0\le |\alpha|+j\leq m, 
	\end{equation}
	where $j=0$ if $l=2,3$; $\alpha_1=0$ if $k=2$; and $\alpha_2=0$ if $k=3$. 
	
	(ii) If $(k,l)=(2,1)$ or $(2,3)$ or $(4,3)$, then for $-1<x<1$,
	\begin{equation}\label{eqA_P_2}
		\left(\ln \frac{1+x}{3}\right)^2|\pt_c^{\alpha} \pt_\gamma^j U^{c,\gamma}_\theta(x)|\le C(m,K), \quad \mbox{for any } 1\le |\alpha|+j\leq m, \alpha_1=0,  
	\end{equation}
	where $j=0$ if $l=3$, and $\alpha_2=0$ if $k=4$. 

	(iii) If $(k,l)=(3,1), (3,2)$ or $(4,2)$, then for $-1<x<1$,
	\begin{equation}\label{eqA_P_3}
		\left(\ln \frac{1-x}{3} \right)^2 |\pt_c^{\alpha} \pt_\gamma^j U^{c,\gamma}_\theta(x)|\le C(m,K), \quad \mbox{for any } 1\le |\alpha|+j\leq m, \alpha_2=0, 
	\end{equation}
	where $j=0$ if $l=2$, and $\alpha_1=0$ if $k=4$. 

	(iv) If $(k,l)=(4,1)$, then  for $-1<x<1$, and for any $1\le |\alpha|+j\leq m$, $\alpha_1=\alpha_2=0$,
	\begin{equation}\label{eqA_P_4}
		\left(\ln \frac{1+x}{3}\right)^2\left(\ln \frac{1-x}{3}\right)^2  |\pt_c^{\alpha} \pt_\gamma^j U^{c,\gamma}_\theta (x)|\le C(m,K). % \quad \mbox{for any } 1\le |\alpha|+j\leq m, \alpha_1=\alpha_2=0. 
	\end{equation}
\end{thm}

To make the above notations clear, we point out  that if $(k,l)=(1,2)$, estimate (\ref{eqA_P_1}) means that for any compact set $K_1\subset J_1$, $\left|\partial_{c}^{\alpha}\left(U_{\theta}^{c,\gamma^+(c)}\right)\right|\le C(m,K_1)$. For other $I_{k,l}$ with $l=2$ or $3$,  the left hand sides in (\ref{eqA_P_1})-(\ref{eqA_P_3}) are interpreted analogously.

\begin{rmk}
   The estimates in Theorem \ref{propA_1} are optimal in each $I_{k,l}$, see examples in Theorem 3.1 in \cite{LLY1}. %
\end{rmk}

\noindent
{\bf Acknowledgment}. The work of the second named author is partially supported by NSF grant DMS-1501004.

%%%%%%%%%%
%***********************
%
%Organization of paper
%
%**********************
%%%%%%%%%%%

%%%%%%%%%%%%%%%%%%%
%\section{Classification of axisymmetric, no-swirl solutions on $\mathbb{S}^2\setminus\{N,S\}$}\label{sec:classification}

\section{Proof of Theorems}

\subsection{Proof of Theorem \ref{thm1_1}, Theorem \ref{thm1_2} and Theorem \ref{thm1_3}}

As mentioned in Section 1, we work with the function $U_{\theta}$ and the variable $x$ given in (\ref{Va}). As explained in \cite{LLY1}, the stationary NSE (\ref{NS}) of (-1)-homogeneous axisymmetric no-swirl solutions can be reduced to (\ref{eq:UthP}) for some constants $c_1,c_2$ and $c_3$. We will show that the existence of solutions of (\ref{eq:UthP}) in $C^1(-1,1)$ depends on the constants $c_1,c_2$ and $c_3$. 

Recall the definitions in (\ref{eq:tau12}) and (\ref{eq:tau12'}).
%\be\label{eq:tau12}
%	\tau_1(c_1) := 2-2\sqrt{1+c_1}, \quad \tau_2(c_1) := 2+2\sqrt{1+c_1}, 
%\ee
%\be\label{eq:tau12'}
%	\tau_1'(c_2) := -2-2\sqrt{1+c_2}, \quad \tau_2'(c_2) := -2+2\sqrt{1+c_2}.  
%\ee

\begin{lem}\label{lem:pre1}
	Let $\delta>0$, $U_\theta\in C^1(-1,-1+\delta)$ satisfy (\ref{eq:UthP}) with $c_1,c_2,c_3 \in \mathbb{R}$. Then $c_1\geq -1$ and $U_\theta(-1) := \lim_{x\to -1^+} U_\theta(x)$ exists and is finite. Moreover, 
	$$
		U_\theta(-1) = \tau_1(c_1) \quad \mbox{or} \quad \tau_2(c_1). 
	$$
\end{lem}
\begin{proof}
	By Proposition 7.1 in \cite{LLY1}, $\lim_{x\to -1^+} U_\theta(x)$ exists and is finite and 
	$$
		\lim_{x\to -1^+}(1+x)U_\theta'(x) = 0. 
	$$ 
	Sending $x$ to $-1$ in (\ref{eq:UthP}) leads to 
	$$
		-2 U_\theta(-1) + \frac{1}{2}U_\theta(-1)^2 = 2 c_1.
	$$
	Thus,  
	$$
		c_1 = \frac{1}{4} \left[ U_\theta(-1) - 2 \right]^2 -1\geq -1, 
	$$
	and $U_\theta(-1) = \tau_1(c_1)$ or $\tau_2(c_1)$. 
\end{proof}

%%%%%%%%%%%%%
\addtocounter{lem}{-1}
\renewcommand{\thelem}{\thesection.\arabic{lem}'}%
\begin{lem}\label{lem:pre1'}
	Let $\delta>0$, $U_\theta\in C^1(1-\delta,1)$ satisfy (\ref{eq:UthP}) with $c_1,c_2,c_3 \in \mathbb{R}$. Then $c_2\geq -1$ and $U_\theta(1) := \lim_{x\to 1^-} U_\theta(x)$ exists and is finite. Moreover, 
	$$
		U_\theta(1)  = \tau_1'(c_2) \quad \mbox{or} \quad \tau_2'(c_2). 
	$$
\end{lem}
\renewcommand{\thelem}{\thesection.\arabic{lem}}%
%%%%%%%%
\begin{proof}
	Consider $\tilde{U}_\theta(x):=-U_\theta(-x)$, and apply Lemma \ref{lem:pre1} to $\tilde{U}_\theta$. 
\end{proof}

\bigskip

%Recall
%
%\be\label{sec2:eq:bar:c3}
%	\bar{c}_3 (c_1,c_2) := -\frac{1}{2} \left( \sqrt{1+c_1} + \sqrt{1+c_2}  \right) \left( \sqrt{1+c_1} + \sqrt{1+c_2}  + 2 \right). 
%\ee
%
%and the set
%\begin{equation}\label{eqJ}
%   J:=\{c | c_1\geq -1, c_2\geq -1,c_3\ge \bar{c}_3(c_1,c_2)\}.
%\end{equation}

\bigskip

\begin{lem}\label{lem2_0_0}
  If $|c|\le A$ for some constant $A>0$, then there exists some constant $C$, depending only on $A$, such that all $C^1$ solutions $U_{\theta}$ of (\ref{eq:UthP}) in $(-1,1)$ satisfy
  \begin{equation*}%\label{eq2_0_0_1}
     |U_{\theta}(x)|\le C, \quad -1<x<1.
  \end{equation*}
\end{lem}
\begin{proof}
   By Lemma \ref{lem:pre1}, there is some $C_1(A)>0$, such that $|U_{\theta}(\pm 1)|\le C_1(A)$ for all solutions $U_{\theta}$ of  (\ref{eq:UthP}) in $(-1,1)$. 
   
   If $\sup_{-1<x<1}|U_{\theta}(x)|\le 8C_1(A)$, the proof is finished. Otherwise, there exists some $\bar{x}\in (-1,1)$ such that $|U_{\theta}(\bar{x})|=\max_{-1\le x\le 1}|U_{\theta}(x)|>8C_1(A)$. We may assume that $U_{\theta}(\bar{x})>8C_1(A)$, since the other case can be handled similarly. Then there exists some $-1<\tilde{x}<\bar{x}$ such that $U_{\theta}(\tilde{x})=\frac{U_{\theta}(\bar{x})}{2}$ and $U_{\theta}'(\tilde{x})\ge 0$. By equation  (\ref{eq:UthP}), we have 
   \[
     -U_{\theta}(\bar{x})+\frac{1}{8}U_{\theta}^2(\bar{x})\le  2\tilde{x}U_{\theta}(\tilde{x})+\frac{1}{2}U_{\theta}^2(\tilde{x})\le P_c(\tilde{x})\le C_2(A).
   \]
 It follows that $U_{\theta}(\bar{x})\le C_3(A)$. The proof is finished.
\end{proof}

%****************rephrase the following paragraph
%
%We first prove the following lemmas on behaviors of solutions of (\ref{eq:UthP}) near $x=-1$ or $x=1$. Roughly speaking,  existence of the real analytic solutions is proved in Lemma \ref{lem:c1c2:LocalSeriesSol} and Lemma \ref{lem:c1c2:LocalSeriesSol'}. Lemma \ref{lem:c1c2:LocalUniqueness} and Lemma \ref{lem:c1c2:LocalUniqueness'} give the uniqueness of solution satisfying $f(-1)>2$ or $f(1)<-2$. 
%While Lemma \ref{lem:c1c2:compare} and Lemma \ref{lem:c1c2:compare'} present some local comparison results. 
%
%***********************************
%While, when $f(-1)=2$ or $f(1)=-2$, it is proved in Lemma \ref{lem:c1c2:LocalBarrier} and \ref{lem:c1c2:LocalBarrier'} that although the solutions might not be unique, but the linear solution behaves like a barrier of other solutions. 

\begin{lem}\label{lem:c1c2:LocalSeriesSol}
	Let $c_1\geq -1$, $\tau = \tau_2(c_1)$ or $\tau = \tau_1(c_1) \not\in \{0, -2, -4, -6, \cdots \}$. Then for every $c_2,c_3\in \mathbb{R}$, there exist $\delta>0$ depending only on an upper bound of $\sum_{i=1}^{3}|c_i|$ and a positive lower bound of $\inf_{k\in \mathbb{N}}|\tau+2k|$,   and a sequence $\{a_n\}_{n=1}^\infty$  such that
	$$
		|a_n| \le \left(\frac{1}{2\delta}\right)^n,%\sum_{n=1}^\infty |a_n| \delta^n < \infty, 
	$$ 
	and 
	\[
		U_{\theta}(x) := \tau + \sum_{n=1}^{\infty} a_n (1+x)^n 
	\]
	is a real analytic solution of (\ref{eq:UthP}) in $(-1,-1+\delta)$.  Moreover, $U_{\theta}$ is the unique real analytic solution of (\ref{eq:UthP}) in $(-1,-1+\delta')$ satisfying $U_{\theta}(-1)=\tau$ for any $0<\delta'\le \delta$.
\end{lem}

%\begin{rmk}
%For $c_1\ge -1$, $\tau_1(c_1)\in  \{0, -2, -4, -6, \cdots \}$, and for every $c_2,c_3\in \mathbb{R}$, there exists no real analytic solution of (\ref{eq:UthP}) near $-1$, satisfying $U_{\theta}(-1)=\tau_1(c_1)$.
%\end{rmk}

\noindent \emph{Proof of Lemma \ref{lem:c1c2:LocalSeriesSol}.}
	Let $s = 1+x$. Rewrite 
	\[
		P_c(x) = 2 c_1 + (-c_1 + c_2 + 2 c_3) (1+x) - c_3(1+x)^2  =: \tilde{c}_1 + \tilde{c}_2 s + \tilde{c}_3 s^2. 
	\]
	Suppose that $U_{\theta} = \tau + \sum_{n=1}^{\infty} a_n s^n$, then $U_{\theta}' = \sum_{n=1}^{\infty}n a_n s^{n-1}$. Plug them into (\ref{eq:UthP}), 
	\[
	\begin{split}
		& \quad \mbox{LHS} \\
		& = s(2-s)\sum_{n=1}^{\infty}n a_n s^{n-1}+2(s-1)(\tau+\sum_{n=1}^{\infty}a_ns^n)+\frac{1}{2}(\tau+\sum_{n=1}^{\infty}a_ns^n)^2\\
		&=\frac{1}{2}\tau^2-2\tau+((2+a_1)\tau)s+\sum_{n=2}^{\infty}[(2n-2+\tau)a_n+(3-n)a_{n-1}+\frac{1}{2}\sum_{k+l=n,k,l\ge 1}a_k a_l]s^n\\
		&= \tilde{c}_1 + \tilde{c}_2 s + \tilde{c}_3 s^2 =\mbox{RHS}.
	\end{split}
	\]
	Compare coefficients, 
	\begin{equation*}
	\begin{array}{lll}
		 n = 0, \quad  \frac{1}{2}\tau^2 - 2 \tau = \tilde{c}_1, \quad & \textrm{so } \tau = 2\pm \sqrt{4+2\tilde{c}_1} = \tau_1(c_1) \mbox{ or } \tau_2(c_1), \\
		 n = 1, \quad (a_1 + 2 ) \tau = \tilde{c}_2, \quad & \textrm{so } a_1=\frac{\tilde{c}_2}{\tau} - 2, \\
		 n = 2, \quad (2 + \tau) a_2 + a_1 + \frac{1}{2} a_1^2 = \tilde{c}_3, & \textrm{ so }a_2=\frac{1}{\tau+2}(\tilde{c}_3 - a_1 - \frac{1}{2} a^2_1).
	\end{array}
	\end{equation*}
	For $n\ge 3$, 
	\[
		(2 n - 2 + \tau) a_n + ( 3 - n ) a_{n-1} + \frac{1}{2} \sum_{k+l=n, k,l\ge 1}a_k a_l=0. 
	\]
	Since for any $n\ge 1$, $\tau\ne -2(n-1)$, 
	\begin{equation}\label{eq2_2_1}
		a_n = - \frac{1}{2n-2+\tau} \left( \frac{1}{2}\sum_{k+l=n,k,l\ge 1}a_k a_l+(3-n)a_{n-1} \right), 
	\end{equation}
	it can be seen that $a_n$ is determined by $a_1,...,a_{n-1}$, thus determined by $c_1,c_2,c_3$ and $\tau$.
	
	{\it Claim}: there exists some $a>0$ large, depending only on an upper bound of $\sum_{i=1}^{3}|c_i|$ and a positive lower bound of $\inf_{k\in \mathbb{N}}|\tau+2k|$, such that 
	\begin{equation*}%\label{eq2_2a}
		|a_n|\le a^n.
	\end{equation*}
	{\it Proof of Claim}: Choose $a>1$ large such that for $1\le n\le 100|\tau|+100$, $|a_n|\le a^n$. 
	
	Now for $n>100|\tau|+100$, suppose that for $1\le k\le n-1$, $|a_k|\le a^k$, then by induction and the recurrence formula (\ref{eq2_2_1}), 
	\[
		|a_n| \le \frac{2}{3(n-1)}|\frac{1}{2}(n-1)a^n+(n-3)a^{n-1}| \le \left(\frac{1}{3}+\frac{2(n-3)}{3(n-1)a}\right)a^n \le a^n.
	\]
 The claim is proved.
	
	So for $\delta<\frac{1}{a}$, 
%	$f$ and $f'$ are both convergent in $(-1,-1+\delta)$. Then 
	$U_{\theta} = \tau + \sum_{n=1}^{\infty} a_n s^n$, with $s=1+x$, is a real analytic solution of (\ref{eq:UthP}) in $(-1,-1+\delta)$. The uniqueness of $U_{\theta}$ is clear from the proof above.
\qed

%\begin{rmk}\label{rmk2_2}
%   In Lemma \ref{lem:c1c2:LocalSeriesSol}, if $\tau=\tau_2(c_1)$ and $|c|\le A$ for some constant $A$, then the $a$ in (\ref{eq2_2a}) satisfies $0<a<C(A)$ where $C(A)$ is a constant depending only on $A$. As a consequence, $\delta\ge \frac{1}{C'(A)}$ for some constant $C'(A)$. 
%\end{rmk}

\addtocounter{lem}{-1}
\renewcommand{\thelem}{\thesection.\arabic{lem}'}%
\begin{lem}\label{lem:c1c2:LocalSeriesSol'}
	Let $c_2\geq -1$, $\tau' = \tau_1'(c_2)$ or $\tau' = \tau_2'(c_2) \not\in \{0, 2, 4, 6, \cdots \}$. Then for every $c_1,c_3\in \mathbb{R}$, there exist $\delta>0$, depending only on an upper bound of $\sum_{i=1}^{3}|c_i|$ and a positive lower bound of $\inf_{k\in \mathbb{N}}|\tau'-2k|$, and a sequence $\{a_n\}_{n=1}^\infty$ such that
	$$
		|a_n| \le \left(\frac{1}{2\delta}\right)^n, %\sum_{n=1}^\infty |a_n| \delta^n < \infty, 
	$$ 
	and 
	\[
		U_{\theta}(x) := \tau' + \sum_{n=1}^{\infty} a_n (1-x)^n 
	\]
	is a real analytic solution of (\ref{eq:UthP}) in $(1-\delta,1)$.  Moreover, $U_\theta$ is the unique real analytic solution of (\ref{eq:UthP}) in $(1-\delta',1)$ satisfying $U_\theta(1)=\tau'$ for any $0<\delta'\le \delta$.
\end{lem}
\renewcommand{\thelem}{\thesection.\arabic{lem}}%

The following two lemmas give some local comparison results.

\begin{lem}\label{lem:c1c2:compare}
	Suppose $0<\delta<2$, $U_{\theta},\tilde{U}_\theta \in C^1(-1,-1+\delta]\cap C^0[-1,-1+\delta]$ satisfy 
	$$
		(1-x^2)U_{\theta}'+ 2x U_{\theta} + \frac{1}{2}U_{\theta}^2 \geq (1-x^2)\tilde{U}_\theta'+ 2x \tilde{U}_\theta + \frac{1}{2}\tilde{U}_\theta^2,  \quad -1<x<-1+\delta. 
	$$
	Suppose also that one of the following two conditions holds. 
	
	(i) $U_{\theta}(-1) \ge  \tilde{U}_\theta(-1)>2$.

	(ii) $U_{\theta}(-1) = \tilde{U}_\theta(-1)=2$, and 
%	$f_1(x) \geq 2 - a_1 (1+x)^\epsilon +o((1+x)^\epsilon)$ for any $\epsilon > 0$, $a_1>0$. 
	\begin{equation}\label{eq:c1c2:f1}
		\limsup_{x\to-1^+} \int_{-1+\delta}^x \frac{-2+U_{\theta}(s)}{1-s^2} ds < +\infty. 
	\end{equation}	
	Then either 
	$$
		U_{\theta} > \tilde{U}_\theta, \quad \mbox{in } (-1, -1+\delta), 
	$$
	or there exists $\delta'\in (0,\delta)$ such that
	$$
		U_{\theta} \equiv \tilde{U}_\theta, \quad \mbox{in } (-1,-1+\delta'). 
	$$
\end{lem}
\begin{proof}
	Let $g = U_{\theta} - \tilde{U}_\theta$, then $g(-1) \ge 0$ and $g$ satisfies 
	\begin{equation}\label{eq:c1c2:compare:g}
		g' + b(x) g  \geq \frac{1}{2(1-x^2)} g^2 \geq 0, \quad \mbox{for all } x\in(-1,-1+\delta), 
	\end{equation}
	where $b(x)$ is given by 
	\begin{equation}\label{eq:c1c2:compare:b}
		b(x) = (1-x^2)^{-1}(2x+U_{\theta}).
	\end{equation}
	Let
	\[
	    w(x) = e^{\int_{-1+\delta}^x b(s) ds} g(x).
	\]
	Then $w$ satisfies, using  (\ref{eq:c1c2:compare:g}), that
	\begin{equation}\label{eq:c1c2:compare:1}
		w'(x) \geq 0 \mbox{ in }(-1,-1+\delta).  
	\end{equation} 
	
	Under condition either (i) or (ii), we have
	\[
	   \limsup_{x\to -1^+} \int_{-1+\delta}^x b(s) ds < +\infty. 
	\]
	Using this and the fact that $g(-1)\ge 0$, we have $\liminf_{x\to-1^+}w(x)\ge 0$. Therefore, using (\ref{eq:c1c2:compare:1}), we have either $w>0$ in $(-1,-1+\delta)$ or there exists a constant $\delta'$, $0<\delta'<\delta$ such that $w\equiv 0$ in $(-1,-1+\delta')$. The lemma is proved. 
\end{proof}

\begin{cor}\label{lem:c1c2:LocalUniqueness}
	For $c_1>-1$, $c_2,c_3\in \mathbb{R}$ and $0<\delta<2$, there exists at most one solution $U_{\theta}$ of (\ref{eq:UthP}) in $C^1(-1,-1+\delta)$ satisfying 
	\[
		\lim_{x\rightarrow -1^+} U_{\theta}(x) = \tau_2(c_1). 
	\]
\end{cor}
\begin{proof}
Since $\tau_2(c_1)>2$ for $c_1>-1$, the uniqueness follows from (i) of Lemma \ref{lem:c1c2:compare}.
\end{proof}
	%Suppose $f_1,f_2$ are two solutions of (\ref{eq:c1c2:f}) in $C^1(-1,-1+\delta]$ satisfying $\lim_{x\to -1^+}f_i(x)=\tau_2(c_1)>2$, $i=1,2$. Then $f_1,f_2$ satisfy 
	%\[
	%   (1-x^2)f_1'+ 2x f_1 + \frac{1}{2}f_1^2 = (1-x^2)f_2'+ 2x f_2 + \frac{1}{2}f_2^2,  \quad -1<x<-1+\delta.
	%\]
	%and $f_1(-1)=f_2(-1)>2$. 
	
	%Applying Lemma \ref{lem:c1c2:compare} we have 
	%\[
	%   f_1\ge f_2,\textrm{  and } f_1\le f_2 \textrm{ in } (-1,-1+\delta).
	%\] 
	%The corollary is proved.

	%Then $g:=f_1-f_2$ satisfies $\lim_{x\to-1^+}g(x)=0$ and 
	%\begin{equation}\label{eq:c1c2:LocalUniqueness}
	%	g' + b(x) g = 0, \quad \mbox{in } (-1,-1+\delta), 
	%\end{equation}
	%where 
	%\[
%	\label{eq:c1c2:LocalUniquenessb}
	%	b(x) = (1-x^2)^{-1} \left( 2x +\frac{1}{2} (f_1+f_2) \right). 
	%\]
	%Since $\lim_{x\to -1^+}f_i(x) = \tau_2 > 2$, 
	%\begin{equation}\label{eq:c1c2:LocalUniqueness2}
	%	\limsup_{x\to -1^+} \int_{-1+\delta}^x b(s) ds < +\infty. 
	%\end{equation}
	
	%Let 
	%\begin{equation}\label{eq:c1c2:LocalUniquenessw}
	%	w(x) = e^{\int_{-1+\delta}^x b(s) ds} g(x).
	%\end{equation}
	%Then we have, using (\ref{eq:c1c2:LocalUniqueness}) and (\ref{eq:c1c2:LocalUniqueness2}) and $ \lim_{x\to -1^+} g(x)=0$ that 
	%$$
	%	\frac{d}{dx} \left[ w(x) \right] = 0, \mbox{ in } (-1,-1+\delta), \quad \quad \lim_{x\to -1^+} w(x) = 0. 
	%$$
	%Thus $w\equiv 0$  and therefore $g\equiv0$. 

Similarly, we have
\addtocounter{lem}{-1}
\renewcommand{\thelem}{\thesection.\arabic{lem}'}%
\begin{lem}\label{lem:c1c2:compare'}
	Suppose $0<\delta<2$, $U_{\theta},\tilde{U}_\theta \in C^1[1-\delta,1)\cap C^0[1-\delta,1]$ satisfy 
	$$
		(1-x^2)U_{\theta}'+ 2x U_{\theta} + \frac{1}{2}U_{\theta}^2 \geq (1-x^2)\tilde{U}_\theta'+ 2x \tilde{U}_\theta + \frac{1}{2}\tilde{U}_\theta^2,  \quad 1-\delta<x<1. 
	$$
	Suppose also that one of the following two conditions holds. 
	
	(i) $U_{\theta}(1) \le \tilde{U}_\theta(1)<-2$,

	(ii) $U_{\theta}(1) = \tilde{U}_\theta(1)=-2$, and 
%	$f_2(x) \leq -2+ a_1 (1-x)^\epsilon+o((1-x)^\epsilon)$ for any $\epsilon>0$, $a_1>0$. 
	\[
		\limsup_{x\to 1^-} \int_{1-\delta}^x \frac{2+U_{\theta}(s)}{1-s^2} ds < +\infty. 
	\]	
	Then either 
	$$
		U_{\theta} < \tilde{U}_\theta, \quad \mbox{in } (1-\delta,1), 
	$$
	or there exists $\delta'\in (0,\delta)$ such that
	$$
		U_{\theta} \equiv \tilde{U}_\theta, \quad \mbox{in } (1-\delta',1). 
	$$
\end{lem}
\renewcommand{\thelem}{\thesection.\arabic{lem}}%

\addtocounter{cor}{-1}
\renewcommand{\thecor}{\thesection.\arabic{cor}'}%
\begin{cor}\label{lem:c1c2:LocalUniqueness'}
	For $c_2>-1$, $c_1,c_3\in \mathbb{R}$ and $0<\delta<2$, there exists at most one solution $U_{\theta}$ of (\ref{eq:UthP}) in $C^1(1-\delta,1)$ satisfying 
	\[
		\lim_{x\rightarrow 1^-} U_{\theta}(x) = \tau_1'(c_2). 
	\]
\end{cor}
\renewcommand{\thecor}{\thesection.\arabic{cor}}

\bigskip

Now we are ready to analyze the global behavior of axisymmetric, no-swirl solutions of NSE (\ref{eq:UthP}) in $(-1,1)$. The behavior of solutions depends closely on parameters $c_1,c_2,c_3\in\mathbb{R}$. 

\bigskip

Recall the definition of $\bar{c}_3(c_1,c_2)$ given by (\ref{sec2:eq:bar:c3}), we have 
%\[
%\bar{c}_3 (c_1,c_2) := -\frac{1}{2} \left( \sqrt{1+c_1} + \sqrt{1+c_2}  \right) \left( \sqrt{1+c_1} + \sqrt{1+c_2}  + 2 \right).
%\]

\begin{lem}\label{lem:c1c2:eqc3bar}
	Suppose $c_1\geq -1$, $c_2\geq -1$, $c_3 = \bar{c}_3(c_1,c_2)$, then $U_{\theta}^*(c_1,c_2)$ given by (\ref{eq1_1_1}) is the unique $C^1$ solution  of (\ref{eq:UthP}) in $(-1,1)$. %Moreover, 
%	$$
%		U_{\theta}^*(c_1,c_2)(x) := (1+\sqrt{1+c_1})(1-x)+(-1-\sqrt{1+c_2})(1+x). 
%	$$
	In particular, 
%	where 
%	$$
%		a = - \left(\sqrt{1+c_1} + \sqrt{1+c_2} + 2 \right), \quad \quad b= \sqrt{1+c_1} - \sqrt{1+c_2}, 
%	$$
%	and
	\[
		U_{\theta}^*(c_1,c_2)(-1) = \tau_2(c_1), \quad U_{\theta}^*(c_1,c_2)(1) = \tau_1'(c_2). 
	\]
\end{lem}
\begin{proof}
	A direct calculation shows that $U_{\theta}^*:=U_{\theta}^*(c_1,c_2)$ is a $C^1$ solution of (\ref{eq:UthP}) in $(-1,1)$. It remains to prove the uniqueness. 
		
	Let $U_{\theta}$ be a $C^1$ solution of (\ref{eq:UthP}) in $(-1,1)$, $U_{\theta}\not\equiv U_{\theta}^*$. By Lemma \ref{lem:pre1} and Lemma \ref{lem:pre1'}, $U_{\theta}$ can be extended as a function in $C^0[-1,1]$, $U_{\theta}(-1)\in\{\tau_1(c_1), \tau_2(c_1)\}$, $U_{\theta}(1)\in\{\tau_1'(c_2), \tau_2'(c_2)\}$. 
	
	%\marginpar{say more?}
	By Corollary \ref{lem:c1c2:LocalUniqueness} and (ii) of Lemma \ref{lem:c1c2:compare}, we know that there exists a constant $0<\delta_1<\frac{1}{2}$ such that $U_{\theta}<U_{\theta}^*$ in $(-1,-1+\delta_1)$. Similarly, by Corollary \ref{lem:c1c2:LocalUniqueness'} and (ii) of Lemma \ref{lem:c1c2:compare'}, we know that there exists a constant $0<\delta_2<\frac{1}{2}$ such that $U_{\theta}>U_{\theta}^*$ in $(1-\delta_2,1)$. 
	
	Therefore, there exists a point $\bar{x}\in (-1+\delta_1,1-\delta_2)$ such that $U_{\theta}(\bar{x})=U_{\theta}^*(\bar{x})$. Standard uniqueness theory of ODE implies that $U_{\theta}\equiv U_{\theta}^*$ in $(-1,1)$. This is a contradiction. 
\end{proof}

\begin{lem}\label{lem:c1c2:<c3bar}
	Suppose $c_1\geq -1$, $c_2\geq -1$, $c_3<\bar{c}_3(c_1,c_2)$, then (\ref{eq:UthP}) has no solution in $C^1(-1,1)$. 
\end{lem}
\begin{proof}
	If $U_{\theta}$ is a $C^1$ solution of (\ref{eq:UthP}) in $(-1,1)$. 
	By Lemma \ref{lem:pre1} and Lemma \ref{lem:pre1'}, $U_{\theta}$ can be extended as a function in $C^0[-1,1]$, $U_{\theta}(-1)\in\{\tau_1(c_1), \tau_2(c_1)\}$, $U_{\theta}(1)\in\{\tau_1'(c_2), \tau_2'(c_2)\}$. 

	By Lemma \ref{lem:c1c2:eqc3bar}, $U_{\theta}^*:=U_{\theta}^*(c_1,c_2)$ is the unique solution of (\ref{eq:UthP}) with $c_3 = \bar{c}_3(c_1,c_2)$. Since $c_3<\bar{c}_3(c_1,c_2)$, $U_{\theta}\not\equiv U_{\theta}^*$ in any open interval in $(-1,1)$. We first assume that $U_{\theta}(\bar{x})>U_{\theta}^*(\bar{x})$ at some point $\bar{x}\in(-1,1)$. Since $c_3<\bar{c}_3(c_1,c_2)$ we have
	\begin{equation}\label{eq:c1c2:<c3bar1}
		(1-x^2)U_{\theta}' + 2xU_{\theta}+\frac{1}{2}U_{\theta}^2 < (1-x^2){U_{\theta}^*}' + 2xU^*_{\theta}+\frac{1}{2}(U^*_{\theta})^2, \quad -1<x<1. 
	\end{equation}

%	
%	
	%\marginpar{say more?}
	Since $U_{\theta}(-1)\leq U_{\theta}^*(-1)$, we have, in view of Lemma \ref{lem:c1c2:compare}, %with $f_1 = f^*$, $f_2=f$, that 
%	from Lemma \ref{lem:c1c2:compare} and the fact that $f\not\equiv f^*$ in any open interval in $(-1,1)$, 
	there exists $\delta>0$ such that $U_{\theta}<U_{\theta}^*$ in $(-1,-1+\delta)$. 
	
	Now with $U_{\theta}(\bar{x})>U_{\theta}^*(\bar{x})$ and $U_{\theta}<U_{\theta}^*$ in $(-1,-1+\delta)$, there exist a point $\xi\in(-1+\delta, \bar{x})$ such that
	$$
		U_{\theta}(\xi) = U_{\theta}^*(\xi), \quad \quad U_{\theta}'(\xi) \geq {U_{\theta}^*}'(\xi), 
	$$
	which contradicts inequality (\ref{eq:c1c2:<c3bar1}) at $\xi$. 
	
	Similar arguments lead to a contradiction when $U_{\theta}(\bar{x})<U_{\theta}^*(\bar{x})$ for some $\bar{x}\in (-1,1)$ by showing $U_{\theta}>U_{\theta}^*$ near $x=1$. The lemma is proved.   
\end{proof}
\begin{lem}\label{lem:c1c2:>c3bar}
%\marginpar{keep unique of not? }
	Suppose $c_1\geq -1$, $c_2\geq -1$, $c_3 > \bar{c}_3(c_1,c_2)$. Let $U_{\theta}^+(c)$ be the  power series solution, obtained in Lemma \ref{lem:c1c2:LocalSeriesSol} with $U_{\theta}^+(c)(-1)=\tau_2(c_1)$, of (\ref{eq:UthP}) in $(-1,-1+\delta)$, 
%	$$
%		f^+_{(c_1,c_2,c_3)}(x) = \tau_2 + \sum_{n=1}^\infty a_n (1+x)^n, 
%	$$
	then $U_{\theta}^+(c)$ can be extended to be a solution of (\ref{eq:UthP}) in $(-1,1)$, and $U_{\theta}^+(c)(1) = \tau_2'(c_2)$. 

	Let $U_{\theta}^-(c)$ be the power series solution, obtained in Lemma \ref{lem:c1c2:LocalSeriesSol'} with $U_{\theta}^-(c)(1)=\tau_1'(c_2)$, of (\ref{eq:UthP}) in $(1-\delta,1)$, 
%	$$
%		f^-_{(c_1,c_2,c_3)}(x) = \tau_1' + \sum_{n=1}^\infty a_n (1-x)^n, 
%	$$
	then $U_{\theta}^-(c)$ can be extended to be a solution of (\ref{eq:UthP}) in $(-1,1)$, and $U_{\theta}^-(c)(-1) = \tau_1(c_1)$. 
	Moreover,  $U_{\theta}^-(c)<U_{\theta}^+(c)$ in $(-1,1)$. 
\end{lem}
\begin{proof}
	We only need to prove that $U_{\theta}^+:=U_{\theta}^+(c)$ can be extended to be a solution of (\ref{eq:UthP}) in $(-1,1)$ and $U_{\theta}^+(1)=\tau_2'(c_2)$, since similar arguments work for $U_{\theta}^-(c)$. 
	
	Standard existence theory of ODE implies that $U_{\theta}^+$ can be extended to the maximal interval of existence, say $(-1,\xi)$, $\xi\in(-1+\delta, 1]$. Since $c_3>\bar{c}_3(c_1,c_2)$, 
%	Recall that $f^*=f^*_{(c_1,c_2)}$ is the unique $C^1$ solution of (\ref{eq:c1c2:f}) in $(-1,1)$ with $c_3 = \bar{c}_3(c_1,c_2)$, $f^*(-1) = \tau_2>2$. 
	we have, with  $U_{\theta}^*:=U_{\theta}^*(c_1,c_2)$,
	\begin{equation*}%\label{eq:c1c2:>c3bar1}
		(1-x^2){U_{\theta}^+}' + 2x{U_{\theta}^+}+\frac{1}{2}(U_{\theta}^+)^2 > (1-x^2){U_{\theta}^*}' + 2xU_{\theta}^*+\frac{1}{2}(U_{\theta}^*)^2, \quad -1<x<\xi.  
	\end{equation*}
	Since $U_{\theta}^+(-1)=U_{\theta}^*(-1) = \tau_2(c_1) \geq 2$, by Lemma \ref{lem:c1c2:compare} and the fact that $U_{\theta}^+, U_{\theta}^*$ can not coincide in any open interval, we have $U_{\theta}^+>U_{\theta}^*$ in $(-1,\xi)$. 
	
	% Then $f^+ > f^*$ in $(-1,\xi)$, since $f^+ > f^*$ in $(-1,-1+\delta)$ and $f^*$ can not reach $f^+$ from underneath. 
%	 since otherwise there exists a point $\bar{x}\in (-1+\delta_1,\xi)$ where
%	$$
%		f^+(\bar{x}) = f^*(\bar{x}), \quad \quad {f^+}'(\bar{x}) \leq {f^*}'(\bar{x}),
%	$$
%	which contradicts inequality (\ref{eq:c1c2:>c3bar1}) at $\bar{x}$. So $f^+ > f^*$ in $(-1,\xi)$. 
	
	If $\xi<1$, since $U_{\theta}^+$ is bounded from below by $U_{\theta}^*$, there exists a sequence of points $\{x_i\}$ satisfying %\rightarrow \xi^-$ satisfying $f^+(x_i)\rightarrow +\infty$. Since $f^+\in C^1(-1,\xi)$, we can find a monotone subsequence of $\{x_i\}$, still denoted by $\{x_i\}$, such that
	\begin{align*}
		& x_1<x_2<x_3<\cdots<\xi, & &\lim_{i\rightarrow \infty} x_i = \xi,  \\
		& U_{\theta}^+(x_1)<U_{\theta}^+(x_2)<U_{\theta}^+(x_3)<\cdots, & &\lim_{i\rightarrow \infty} U_{\theta}^+(x_i) = +\infty.   
	\end{align*}
	Then, in each interval $(x_i,x_{i+1})$, we can find a point $y_i$ such that 
	$$
		x_i < y_i < x_{i+1}, \quad U_{\theta}^+(y_i) \geq U_{\theta}^+(x_i),  \quad {U_{\theta}^+}'(y_i) \geq 0. 
	$$
	Taking $x=y_i$ in equation (\ref{eq:UthP}), and sending $i$ to infinity, we obtain a contradiction. So $\xi=1$. By Lemma \ref{lem:pre1}, $\lim_{x\to1^+}U_{\theta}^+(x)$ exists and is finite. 
	
	We have extended $U_{\theta}^+$ to be a solution of (\ref{eq:UthP}) in $C^1(-1,1)\cap C^0[-1,1]$ and $U_{\theta}^+>U_{\theta}^*$ in $(-1,1)$. 
	
	Similarly, $U_{\theta}^-$ can be extended to $C^0[-1,1]$, and $U_{\theta}^-<U_{\theta}^*<U_{\theta}^+$ in $(-1,1)$. 
	
	By Lemma \ref{lem:pre1'}, $U_{\theta}^+(1)\in\{\tau_1'(c_2),\tau_2'(c_2)\}$. If $c_2=-1$, $\tau_1'(c_2)=\tau_2'(c_2)$, so $U_{\theta}^+(1)=\tau_2'(c_2)$. If $c_2>-1$, since $U_{\theta}^-(1)=\tau_1'(c_2)$ and $U_{\theta}^+>U_{\theta}^-$ in $(-1,1)$, by Corollary \ref{lem:c1c2:LocalUniqueness'}, we have $U_{\theta}^+(1) = \tau_2'(c_2)$.  Similarly, $U_{\theta}^-(-1)=\tau_1(c_1)$. Lemma \ref{lem:c1c2:>c3bar} is proved. 
\end{proof}

\begin{lem}\label{lem:c1c2:between}
	Suppose $c_1\geq -1$, $c_2\geq -1$, $c_3 > \bar{c}_3(c_1,c_2)$,  %Let $U_{\theta}^+(c), U_{\theta}^-(c)$ be the unique $C^1$ solutions of (\ref{eq:UthP}) in $(-1,1)$ obtained in Lemma \ref{lem:c1c2:>c3bar}. 
	then any $C^1$ solution $U_{\theta}$ of (\ref{eq:UthP}) in $(-1,1)$ other than $U_{\theta}^{\pm}(c)$ satisfies
	$$
		U_{\theta}^-(c)<U_{\theta}<U_{\theta}^+(c), \quad \mbox{in } (-1,1), 
	$$
	$$
		U_{\theta}(-1) = \tau_1(c_1), \quad U_{\theta}(1) = \tau_2'(c_2). 
	$$
%	%	$$
\end{lem}

\begin{proof}
	By Lemma \ref{lem:pre1} and Lemma \ref{lem:pre1'},  $U_{\theta}$ can be extended to $C^0[-1,1]$ with $U_{\theta}(-1)=\tau_1(c_1)$ or $\tau_2(c_1)$, and $U_{\theta}(1)=\tau_1'(c_2)$ or $\tau'_2(c_2)$.  
	
	We only need to prove $U_{\theta}<U_{\theta}^+(c)$ in $(-1,1)$ and $U_{\theta}(-1)=\tau_1(c_1)$, since similar arguments imply that $U_{\theta}>U_{\theta}^-(c)$ in $(-1,1)$ and $U_{\theta}(1)=\tau_2'(c_2)$.  
	
	From the standard uniqueness theory of ODE, we know that the graph of $U_{\theta}$ and $U_{\theta}^+(c)$ can not intersect in $(-1,1)$. So we either have $U_{\theta}<U_{\theta}^+(c)$ in $(-1,1)$ or $U_{\theta}>U_{\theta}^+(c)$ in $(-1,1)$. 
	
	If $U_{\theta}>U_{\theta}^+(c)$ in $(-1,-1+\delta)$, then, by Lemma \ref{lem:pre1}, $U_{\theta}(-1)=U_{\theta}^+(c)(-1)=\tau_2(c_1)\geq 2$. Note that $U_{\theta}^+(c)$ satisfies (\ref{eq:c1c2:f1}), we can apply Lemma \ref{lem:c1c2:compare} % with $f_1=f^+_c$, $f_2=f$
	  to obtain $U_{\theta}\leq U_{\theta}^+(c)$, a contradiction. So $U_{\theta}<U_{\theta}^+(c)$ in $(-1,1)$. 
	
	If $\tau_1(c_1)<\tau_2(c_1)$, the uniqueness result Corollary \ref{lem:c1c2:LocalUniqueness} implies that $U_{\theta}(-1)=\tau_1(c_1)$. If $\tau_1(c_1)=\tau_2(c_1)$, we again have $U_{\theta}(-1)=\tau_1(c_1)$. Lemma \ref{lem:c1c2:between} is proved. 
\end{proof}

\noindent \emph{Proof of Theorem \ref{thm1_1}}: For $c\in J$, if $c_3=\bar{c}_3$, by Lemma \ref{lem:c1c2:eqc3bar},  $U^*_{\theta}(c_1,c_2)$ in (\ref{eq1_1_1}) 
%\[
%    U^*_{\theta}(c):=(1+\sqrt{1+c_1})(1-x)+(-1-\sqrt{1+c_2})(1+x),
%\]
is the unique solution of (\ref{eq:UthP}) in $(-1,1)$.

If $c_3>\bar{c}_3$, let $U^{+}_{\theta}(c)$ and $U^{-}_{\theta}(c)$ be the functions in Lemma \ref{lem:c1c2:>c3bar}. By Lemma \ref{lem:c1c2:LocalSeriesSol}, Lemma \ref{lem:c1c2:LocalSeriesSol'}, Lemma \ref{lem:c1c2:>c3bar} and Lemma \ref{lem:c1c2:between}, $U^{\pm}_\theta(c)\in C^{\infty}(-1,1)\cap C^0[-1,1]$ satisfy (\ref{eq:UthP}) in $(-1,1)$, and $U^-_{\theta}(c)<U^+_{\theta}(c)$.  Moreover, $U^{-}_{\theta}(c)\le U_{\theta}\le U^+_{\theta}(c)$ for any solution $U_{\theta}$ of (\ref{eq:UthP}) in $(-1,1)$.

Now we prove the continuity of $U^+_{\theta}(c)(x)$ in $(c,x)$, the same arguments applies to $U^-_{\theta}$.

For every $(\hat{c},\hat{x})\in J\times [-1,1]$, we prove the continuity of $U_{\theta}^+$ at $(\hat{c},\hat{x})$. By Lemma \ref{lem:c1c2:LocalSeriesSol}, there exists some $\delta>0$,  such that $U^+_{\theta}(c)(x)$ is continuous in $(B_1(\hat{c})\cap J)\times [-1,-1+\delta]$, where $B_1(\hat{c})$ is the unit ball in $\mathbb{R}^3$ centered at $\hat{c}$.

%We know from Lemma \ref{lem:c1c2:LocalSeriesSol} and Lemma \ref{lem:c1c2:>c3bar}, for each $c\in J$, there exists $\delta(c)>0$, such that 
%\[
%   U^+_{\theta}(c)=\tau_2(c_1)+\sum_{n=1}^{\infty}a_n(c)(1+x)^n, \quad -1<x<-1+\delta(c).
%\]
%
%For any $\hat{c}\in J$, by Lemma \ref{lem:c1c2:LocalSeriesSol},  % and Remark \ref{rmk2_2}, 
%we know that there exists a $\delta>0$, depending only on $\hat{c}$, such that $\delta<\delta(c)$ for all $|c-\hat{c}|\le 1$. This means the series is uniform convergence for $x\in (-1,-1+\delta)$ and $|c-\hat{c}|\le 1$. From the expression of $a_n(c)$ in the proof of Lemma \ref{lem:c1c2:LocalSeriesSol}, we know that $a_n(c)$ is continuous in $c$ for all $c\in J$. So $U^+_{\theta}(c)(-1+\frac{\delta}{2})$ is continuous  at $\hat{c}$. 

%Next, notice that $U^+_{\theta}(c)$ is the solution of the initial value problem 
%    \[
%   \left\{\begin{split}
%      & (1-x^2)f'+2xf+\frac{1}{2}f^2=P_c(x):=c_1(1-x)+c_2(1+x)+c_3(1-x^2), \quad -1<x<1,\\
%      & f(-1+\frac{\delta}{2})=U^+_{\theta}(c)(-1+\frac{\delta}{2}).
%   \end{split}
%   \right.
% \]
%By standard ODE theory we know that $\gamma^+(c)=U^+_{\theta}(c)(0)$ is continuous at $\hat{c}$. Thus $\gamma^+(c)\in C^0(J)$. Similarly, we also have $\gamma^-(c)\in C^0(J)$.

 Consider
 \begin{equation}\label{eq_cor2_7}
   \left\{\begin{split}
      & (1-x^2)U_{\theta}'+2xU_{\theta}+\frac{1}{2}U_{\theta}^2=P_c(x)=c_1(1-x)+c_2(1+x)+c_3(1-x^2),\\% \quad -1<x<1,\\
      &U_{\theta}(-1+\frac{\delta}{2})=a,  %f^+_c(-1+\frac{\delta}{2}).
   \end{split}
   \right.
 \end{equation}
 for $a$ close to $a_0:=U_{\theta}^+(c)(-1+\frac{\delta}{2})$.
% has a $C^\infty$ solution $U_{\theta}^+(c)$ in $(-1,1)$ if $a=a_0:=U_{\theta}^+(c)(-1+\frac{\delta}{2})$.
 
 By standard ODE theories, for any $0<\epsilon<2-\delta$, there exists some positive constants $\mu$,  such that  $U^+_{\theta}\in C((a_0-\mu,a_0+\mu)\times (B_1(\hat{c})\cap J)\times [-1+\frac{\delta}{4}, 1-\epsilon])$. %of (\ref{eq_cor2_7}), and 
% \[
%     |\partial_{a}^{\beta}\partial_c^{\alpha}U_{\theta}|\le C, \quad  |\beta|, |\alpha|\le m, c\in K, -1+\frac{\delta}{4}<x<1-\epsilon.
% \]
% It follows,   also in view of (\ref{eq2_7_00}), that $U^+_{\theta}(c)=U_{\theta}|_{a=a_0}$ satisfies (\ref{eqcor2_7_1}).

%By Lemma \ref{lem:c1c2:between}, we have that $U^{-}_{\theta}(c)\le U_{\theta}\le U^+_{\theta}(c)$ for any solution $U_{\theta}$ of (\ref{eq:UthP}) in $(-1,1)$.

%Theorem \ref{thm1_1} is established.

The continuity of $U^+_{\theta}(c)(x)$ at $\hat{x}=1$ follows from Lemma \ref{lemAp_1'}, which will be given later.% let $(c^i,x^i)\to (\hat{c},1)$. Let $g_i=U^+_{\theta}(c^i)-U^+_{\theta}(\hat{c})$, then 
%\[
%    g_i'+\frac{2x+\frac{1}{2}(U^+_{\theta}(c^i)+U^+_{\theta}(\hat{c}))}{1-x^2}g_i=\frac{h_i}{1-x^2}:=\frac{P_{c^i}-P_c}{1-x^2}.
%\]
%It follows that 
%\[
%   g_i=e^{-w_i}\int_{0}^{x}e^{w_i}\frac{h_i}{1-s^2}ds+g_i(0)e^{-w_i}
%\]
%where $w_i:=\int_{0}^{x}\frac{2s+\frac{1}{2}(U^+_{\theta}(c^i)+U^+_{\theta}(\hat{c}))}{1-s^2}ds$.
\qed

\bigskip

\noindent \emph{Proof of Theorem \ref{thm1_2}}: Let $(c,\gamma)\in I$. If $c_3=\bar{c}_3$, then $\gamma=\gamma^+(c)=\gamma^-(c)$ by Theorem \ref{thm1_1}, $U^{c,\gamma}_{\theta}:=U^{\pm}_{\theta}(c)$ given by (\ref{eq1_1_1}) is the unique solution of (\ref{eq:UthP}) satisfying $U^{c,\gamma}_{\theta}(0)=\gamma$.

If $c_3 > \bar{c}_3(c_1,c_2)$, and $\gamma=\gamma^{\pm}(c)$, then $U^{c,\gamma}_{\theta}:=U^{\pm}_{\theta}(c)$ is the unique solution of (\ref{eq:UthP}) satisfying $U^{c,\gamma}_{\theta}(0)=\gamma$.
%\marginpar{'linear'??}
%Let $\zeta_c: [0,1]\rightarrow [f^-_c(0),f^+_c(0)]$ be a monotonically increasing linear map, which is 1-1 and onto. 

For $c_3 > \bar{c}_3(c_1,c_2)$, and $\gamma^-(c)<\gamma<\gamma^+(c)$,  let  $U_{\theta}^{c,\gamma}$ be the unique local solution of (\ref{eq:UthP}) satisfying $U_{\theta}^{c,\gamma}(0)=\gamma$.  By standard ODE theory, $U_{\theta}^{c,\gamma}$ can be extended to a $C^{\infty}$ solution in $(-1,1)$ satisfying $U_{\theta}^-(c)<U_{\theta}^{c,\gamma}<U_{\theta}^+(c)$.

By Lemma \ref{lem:pre1} and Lemma \ref{lem:pre1'}, $U_{\theta}^{c,\gamma}$ can be extend as a function in $C^0[-1,1]$. 

To complete the proof of Theorem \ref{thm1_2}, it remains to show that $\{U^{c,\gamma}_{\theta}|(c,\gamma)\in J\}$ are all the solutions.

For $c\in \mathbb{R}^3$, let $U_{\theta}$ be a solution of (\ref{eq:UthP}) in $(-1,1)$, By Lemma \ref{lem:pre1} and Lemma \ref{lem:pre1'}, $c_1\ge -1$ and $c_2\ge -1$. Then by Lemma \ref{lem:c1c2:<c3bar}, $c_3\ge \bar{c}_3$. So $c\in J$. By Theorem \ref{thm1_1}, we have $U^-_{\theta}(c)\le U_{\theta}\le U_{\theta}^+(c)$.  So $\gamma:=U_{\theta}(0)$ satisfies $\gamma^-(c)\le\gamma\le\gamma^+(c)$, and $U_{\theta}=U^{c,\gamma}_{\theta}$.

%Theorem \ref{thm1_2} is proved.
\qed

%\marginpar{added one sentence here}
%It is obvious that $f^-_c = f^{c,\gamma^-}$ and $f^+_c = f^{c,\gamma^+}$. 
%If $\gamma^-(c) < \gamma < \gamma^+(c)$, Lemma \ref{lem:c1c2:between} implies that $f^{c,\gamma}$ can be extended to be a solution of (\ref{eq:c1c2:f}) in $C^1(-1,1)$ satisfying $f^{c,\gamma}(-1)=\tau_1(c_1)$ and $f^{c,\gamma}(1)=\tau'_2(c_2)$. 

\bigskip

\begin{lem}\label{lem:c1c2:foliate}
	Suppose $c_1\geq -1$, $c_2\geq -1$, $c_3 > \bar{c}_3(c_1,c_2)$, then $\gamma^-(c)<\gamma^+(c)$, and the graphs
	$$
		K_1(\gamma):=\{ \left( x, U_{\theta}^{c,\gamma}(x) \right) \mid -1<x<1 \}, \quad \gamma^-(c) \le \gamma \le \gamma^+(c), 
	$$
	foliate the set 
	$$
		K_2:=\{ (x,y) \mid -1<x<1, U_{\theta}^{c,\gamma^-}(x) \le y \le U_{\theta}^{c,\gamma^+}(x) \} 
	$$
	in the sense that for any $\gamma,\gamma'\in \mathbb{R}$, $\gamma^-(c) \leq \gamma < \gamma' \leq \gamma^+(c)$, $U_{\theta}^{c,\gamma} < U_{\theta}^{c,\gamma'}$ in $(-1,1)$ and $K_2 = \bigcup_{\gamma^-(c)\le\gamma \le \gamma^+(c)} K_1(\gamma)$. 
	Moreover, $U_{\theta}^{c,\gamma}$ is a continuous function of  $(c, \gamma,x)$ in $J \times [\gamma^-(c),\gamma^+(c)]\times(-1,1)$. 
%	Moreover, $f^{c,\gamma}$ is continuous for $(\gamma,x)\in([0,1]\times(-1,1))$. 
\end{lem}
\begin{proof}
	By standard uniqueness theories of ODE,  
	$$
		U_{\theta}^{c,\gamma^-}< U_{\theta}^{c,\gamma}< U_{\theta}^{c,\gamma'} < U_{\theta}^{c,\gamma^+} \quad \mbox{in }(-1,1),\quad \textrm{ for } \gamma^-(c)<\gamma<\gamma'<\gamma^+(c). 
	$$
%	$$
%		f^{c,0} \leq f^{c,\gamma}< f^{c,\gamma'} \leq f^{c,1} \quad \mbox{in }(-1,1),\quad \quad 0<\gamma<\gamma'<1. 
%	$$
	
	It is obvious that $K_1(\gamma)\subseteq K_2$. On the other hand, let $(x_0,y_0)\in K_2$, so $-1<x_0<1$ and 
	$U_{\theta}^{c,\gamma^-}(x_0)<y_0<U_{\theta}^{c,\gamma^+}(x_0)$. 
%	$f^{c,0}(x_0)<y_0<f^{c,1}(x_0)$. 
	By standard existence and uniqueness theories of ODE, there exists a $C^1$ solution $U_{\theta}$ of (\ref{eq:UthP}) in $(-1,1)$ satisfying $U_{\theta}(x_0)=y_0$ and 
	$U_{\theta}^{c,\gamma^-}<U_{\theta}<U_{\theta}^{c,\gamma^+}$ in $(-1,1)$. 
%	$f^{c,0}<f<f^{c,1}$ in $(-1,1)$. 
	In particular, 
	$$
		\gamma^- = U_{\theta}^{c,\gamma^-}(0)<U_{\theta}(0)<U_{\theta}^{c,\gamma^+}(0)=\gamma^+,
	$$
%	$$
%		f^-_c(0) = \zeta_c(0) = f^{c,0}<f(0)<f^{c,1}=\zeta_c(1)=f^-_c(1),
%	$$
	$U_{\theta} = U_{\theta}^{c,\gamma}$ with 
	$\gamma=U_{\theta}(0)$
%	$\gamma=\zeta_c^{-1} (f(0))$
	 and therefore $x_0,y_0\in K_1(\gamma)$. We have proved that $K_2 = \bigcup_{\gamma^-\le\gamma\le \gamma^+} K_1(\gamma)$. 
	
	The continuity of $U_{\theta}^{c,\gamma}$ for $(c, \gamma,x)$ in $J \times [\gamma^-(c),\gamma^+(c)]\times (-1,1)$ can be derived from (\ref{eq:UthP}), and the continuous dependence of ODE on its boundary conditions.
%%%%%%%%%%%%%%%%%%%%%%%%%%
%	
%	Now we prove that $K_2 \subseteq \bigcup_{\gamma\in (\gamma^-,\gamma^+)} K_1(\gamma)$. Suppose the contrary, there exists a point $(x_0,y_0)\in K_2$, $(x_0,y_0)\not\in K_1(\gamma)$ for any $\gamma\in (\gamma^-,\gamma^+)$. 
%	So $-1<x_0<1$ and $f^-_c(x_0)<y_0<f^+_c(x_0)$. 
%	The ODE (\ref{eq:c1c2:f}) is locally uniquely solvable under condition $f(x_0)=y_0$. Denote the solution by $f$, then we can extend $f$ to the maximal interval of existence. Since $f, f^\pm_c$ can not intersect with each other in $(-1,1)$ and $f^-_c(x_0)<y_0<f^+_c(x_0)$, we have $f\in C^1(-1,1)$ and $f^-_c<f<f^+_c$ in $(-1,1)$. Therefore, the graph of $f$ must intersect with the line $x=0$ at a point $(0, f(0))$ with $f^-_c(0)<f(0)<f^+_c(0)$. Thus $(x_0,y_0)$ lies in $K_1(f(0))$, which is a contradiction. The lemma is proved. 
\end{proof}

\noindent \emph{Proof of Theorem \ref{thm1_3}}:
Theorem \ref{thm1_3} follows from Lemma \ref{lem:c1c2:eqc3bar} - Lemma \ref{lem:c1c2:foliate}. 
\qed

\bigskip

\bigskip

%%%%%%%%%%%%%%%%%

\bigskip

\subsection{Proof of Theorem \ref{thm:c1c2} and Theorem \ref{propA_1}}

%\begin{thm}\label{thm:c1c2}
%$\gamma^+(c)$ is  in $C^{\infty}(J\setminus\{c| c_1=-1\})$, and $\gamma^+(-1,c_2,c_3)$ is  in $C^{\infty}(J\cap\{c| c_1=-1\})$  as a function of $(c_2,c_3)$.
%
%$\gamma^-(c)$ is  in $C^{\infty}(J\setminus\{c| c_2=-1\})$, and $\gamma^-(c_1,-1,c_3)$ is  in $C^{\infty}(J\cap\{c| c_2=-1\})$  as a function of $(c_1,c_3)$.
% \end{thm}
%
%
%\bigskip

%Denote $H_i$, $1\le i\le 6$ to be the following subsets of $D$:
%\[
%    H_1:=\{c\in J | c_1>-1,c_2\ge -1,c_3>\bar{c}_3\},\quad H_2:=\{c\in J | c_1=-1,c_2\ge -1,c_3>\bar{c}_3\},
%\]
%and let 
%\[
%\begin{split}
%    & H_3:=\{c\in J | c_1>-1,c_2> -1,c_3\ge \bar{c}_3\},\quad H_4:=\{c\in J | c_1=-1,c_2> -1,c_3\ge \bar{c}_3\},\\
%    & H_5:=\{c\in J | c_1>-1,c_2= -1,c_3\ge \bar{c}_3\},\quad H_6:=\{c\in J | c_1=-1,c_2= -1,c_3\ge \bar{c}_3\}
%    \end{split}
%\]
In the following context,   in $J\cap\{c\in J\mid c_1=-1\}$, $U_{\theta}^+(c)=U_{\theta}^+(-1,c_2,c_3)$ is viewed as a function of $(c_2,c_3)$, and $\partial_c^{\alpha}U_{\theta}^+(c)(x)$ means $\partial_{(c_2,c_3)}^{\alpha}U_{\theta}^+(c)(x)$.  In $ J\cap\{c\mid c_2=-1\}$, $U_{\theta}^-(c)=U_{\theta}^-(c_1,-1,c_3)$ is viewed as a function of $(c_1,c_3)$, and $\partial_c^{\alpha}U_{\theta}^-(c)(x)$ means $\partial_{(c_1,c_3)}^{\alpha}U_{\theta}^-(c)(x)$. 
\begin{lem}\label{lem2_gamma_1}
    %Let $U_{\theta}^+(c)$ be the power series solution, obtained in Lemma \ref{lem:c1c2:LocalSeriesSol} with $U_{\theta}^+(c)(-1)=\tau_2(c_1)$.  
    For any integer $m\ge 0$,  and any compact subset $K$ contained in either $J\setminus\{c \mid c_1=-1\}$ or $J\cap\{c\in J\mid c_1=-1\}$, there exist some positive constants $\delta$ and $C$, depending only on $m$ and $K$, such that $U_{\theta}^+(c)\in C^m(K\times (-1,-1+\delta))$, and %. Moreover, there exists some constant $C$, depending only on $m,K$ and $\delta$, such that for any multi-index $\alpha$ satisfying $|\alpha|\le m$, 
    \begin{equation}\label{eq2_7_00}
        |\partial_c^{\alpha}U_{\theta}^+(c)(x)\mid \le C, \quad  x\in (-1,-1+\delta), c\in K, |\alpha|\le m.
    \end{equation}
\end{lem}
\begin{proof}
Let $\alpha=(\alpha^1,\alpha^2,\alpha^3)$ denote a multi-index where $\alpha^i\ge 0$, $i=1,2,3$. The partial derivative $\partial^{\alpha}=\partial^{\alpha_1}_{c_1}\partial^{\alpha_2}_{c_2}\partial^{\alpha_3}_{c_3}$ and the absolute value $|\alpha|=\sum_{i=1}^{3}\alpha^i$.

By Lemma \ref{lem:c1c2:LocalSeriesSol} and its proof,  there exists $\delta>0$, depending only on $K$, such that for  $c\in K$, $U_{\theta}^+(c)$ can be expressed as 
\[
   U_{\theta}^+(c)(x)=\tau+\sum_{n=1}^{\infty}a_n(1+x)^n, -1<x<-1+\delta,
\]
where
\begin{equation*}%\label{eq2_7_5}
   \tau=\tau_2(c_1)=2+2\sqrt{1+c_1},
\end{equation*}
\begin{equation}\label{eq2_7_6}
   a_1=\frac{-c_1+c_2+2c_3}{\tau}-2,\quad a_2=-\frac{1}{\tau+2}(c_3+a_1+\frac{1}{2}a_1^2),
\end{equation}
\begin{equation}\label{eq2_gamma_0}
   a_n=-\frac{1}{2n-2+\tau}\left(\frac{1}{2}\sum_{k+l=n,k,l\ge 1}a_ka_l+(3-n)a_{n-1}\right),\quad n\ge 3,
\end{equation}
and 
\begin{equation}\label{eq2_7_6b}
  |a_n|\le \left(\frac{1}{2\delta}\right)^n.
\end{equation}
Estimate (\ref{eq2_7_6b}) guarantees that the power series expansion of $U^+_{\theta}(c)(x)$ is uniformly convergent in $(-1,-1+\delta)$.

By the above expressions and relations it can be seen that $\tau(c)$ and $a_n(c)$ are all $C^{\infty}$ functions of $c$ in $J$. So to prove the lemma, we just need to show that there exists some $\delta'>0$, depending only on $m$ and $K$, such that for any multi-index $\alpha$ satisfying $1\le |\alpha|\le m$, the series
\begin{equation}\label{eq2_7_7}
     \frac{\partial^{\alpha}\tau}{\partial c^{\alpha}}+\sum_{n=1}^{\infty}\frac{\partial^{\alpha} a_n}{\partial c^{\alpha}}(1+x)^n
\end{equation}
is absolutely convergent in $(-1,-1+\delta')$ uniformly for $c\in K$.
\bigskip

\textbf{Case 1}: $K\subset J\setminus\{c\mid  c_1=-1\}$.

Let $C(m,K)$ be a constant depending only on $m$ and $K$ which may vary from line to line. If $K$ is a compact set in $J\setminus\{c\mid  c_1=-1\}$, there exists some constant $\delta_1(K)>0$,  such that $4+4c_1\ge \delta_1(K)$.  Using this, (\ref{eq2_7_6}), (\ref{eq2_gamma_0}), and the fact that $\tau>2$, we have

\begin{equation}\label{eq2_7_1}
  \left|\frac{\partial^{\alpha} \tau}{\partial c^{\alpha}}\right|\le C(m,K), \quad \left|\frac{\partial^{\alpha} a_n}{\partial c^{\alpha}}\right|\le C(m,K), \quad \forall 1\le n\le 2, 0\le |\alpha|\le m, c\in K.
\end{equation}

Next, let $g_n(c):=\frac{1}{2n-2+\tau}$. By the above estimates and the fact that $\tau>2$, we have 
\begin{equation}\label{eq2_7_2}
   \left|\frac{\partial^{\alpha}}{\partial c^{\alpha}}g_n(c)\right|\le \frac{C(m,K)}{n},  \quad \textrm{ for all }1\le |\alpha|\le m,  c\in K, \textrm{ and }n \ge 1.
\end{equation}

To prove the existence of $\delta'$ such that the series in (\ref{eq2_7_7}) is convergent for all $1\le |\alpha|\le m$ uniformly in $K$, we will only need to show the following:

\medskip

\noindent {\it Claim}: there exists some $a>0$, depending only on $m$ and $K$, such that
\[
   (P_n): \quad |\partial^{\alpha} a_n(c)|\le a^{n(|\alpha|+1)}, \quad \textrm{ for }1\le |\alpha|\le m, \textrm{ and } c\in K
\]
holds for all $n\ge 1$.

\noindent {\it Proof of Claim}:  We prove it by induction on $n$. Let  $a$ be a constant to be determined in the proof. 

%Lemma \ref{lem:c1c2:LocalSeriesSol} and its proof, and
By  estimate (\ref{eq2_7_1}), there exists some constant $\bar{a}$, depending only on $m$ and $K$, such that for all $a\ge \bar{a}$, $(P_1)$ and $(P_2)$ hold. We may assume that $\bar{a}\ge \frac{1}{2\delta}$ so that we know from (\ref{eq2_7_6b}) that 
\begin{equation}\label{eq2_7_3a}
   |a_n(c)|\le \bar{a}^n,
   \end{equation}
    for all $c\in K$ and $n\ge 1$. 

Now for $n\ge 3$,  suppose that for some $a\ge \bar{a}$, $(P_k)$ holds for all $1\le k\le n-1$.
 
Let $Q_n(c):=\sum_{k+l=n,k,l\ge 1}a_ka_l$. Then (\ref{eq2_gamma_0}) can be written as 
\[
    a_n=-\frac{1}{2}g_nQ_n+(n-3)g_na_{n-1}.
\]
So
\begin{equation}\label{eq2_7_3}
   \partial^{\alpha}a_n=-\frac{1}{2} \partial^{\alpha}(g_nQ_n)+(n-3) \partial^{\alpha}(g_na_{n-1}).
\end{equation}
Using (\ref{eq2_7_2}), by computation we have
\[
    | \partial^{\alpha}(g_nQ_n)|
	 \le \frac{C(m,K)}{n}\max_{\alpha_1\le \alpha}|\partial^{\alpha_1}Q_n|.
		\]

Let $a\ge \bar{a}$, using the definition of $Q_n(c)$, by induction we have that,
\begin{equation}\label{eq2_7_13}
\begin{split}
   | \partial^{\alpha}(g_nQ_n)|
	 &\le \frac{C(m,K)}{n}\max_{\alpha_1\le \alpha}\sum_{k+l=n,k,l\ge 1}\max_{\alpha_2\le \alpha_1}|\partial^{\alpha_2}a_k||\partial^{\alpha_1-\alpha_2}a_l|\\
    & \le C(m,K)\max_{\alpha_1\le \alpha}\max_{\alpha_2\le \alpha_1}\max_{k+l=n,k,l\ge 1}a^{k(|\alpha_2|+1)}a^{l(|\alpha_1-\alpha_2|+1)}\\
    & \le C(m,K)a^{n(|\alpha|+1)-|\alpha|}.
   \end{split}
\end{equation}

Similarly, by (\ref{eq2_7_2}), (\ref{eq2_7_3a}) and the induction hypothesis, we have
\begin{equation}\label{eq2_7_14}
    | \partial^{\alpha}(g_na_{n-1})|  \le \frac{C(m,K)}{n}a^{(n-1)(|\alpha|+1)}.
\end{equation}

Plug (\ref{eq2_7_13}) and (\ref{eq2_7_14}) in (\ref{eq2_7_3}), we have that for $|\alpha|\ge 1$, 
\[
     |\partial^{\alpha} a_n|  \le C(m,K)a^{n(|\alpha|+1)-1}.
\]

%Let  $a>\max\{a_0,C(m,K),a_1\}$. We have 
If from the beginning we use $a=\max\{\bar{a},C(m,K)\}$ for the induction hypothesis, we have
\[
      |\partial^{\alpha} a_n|  \le a^{n(|\alpha|+1)}.
\]

  So the claim is true for all $n$. The lemma is proved for $K\subset J\setminus\{c\mid c_1=-1\}$.
  \bigskip
  
  \textbf{Case 2}:  $K\subset J\cap\{c\mid  c_1=-1\}$. 
  
  In this case $\tau=2$ and $g_n(c)$ is a constant in $K$. By similar arguments as in Case 1, we have the same estimate for $a_n$ and the proof is finished.  
%  \textbf{Case 3}: $K\subset H_i$, $3\le i\le 6$.
%  
%   By the definition of $\bar{c}_3(c_1,c_2)$, we know $\bar{c}_3(c_1,c_2)$ is smooth in $c_1,c_2$ in each of $\{c_1>-1,c_2>-1\}$, $\{c_1=-1,c_2>-1\}$, $\{c_1>-1,c_2=-1\}$, and $\{c_1=-1,c_2=-1\}$
%   By similar arguments as in Case 1, we have the same estimate for $\tau$ and $a_n$ and the proof is finished.
 \end{proof}

	\begin{cor}\label{cor2_7}
   For any  $K\subset J\setminus\{c\mid c_1=-1\}$ or $J\cap\{c \mid c_1=-1\}$, $U_{\theta}^+(c)\in C^{\infty}(K\times (-1,1))$. Moreover, for any $\epsilon>0$, $m\in \mathbb{N}$, there exists some positive constant $C$, depending only on $m$, $K$, and $\epsilon$,  such that
   \begin{equation}\label{eqcor2_7_1}
      ||\partial^{\alpha}_cU_{\theta}^+(c)||_{L^{\infty}(-1, 1-\epsilon)}\le C(m,K,\epsilon), \quad  0\le |\alpha|\le m.
   \end{equation}
\end{cor}	
%\begin{cor}\label{cor2_7}
%   For any $x_0\in (-1,1)$, and any $K\subset J\setminus\{c\in J| c_1=-1\}$ or $J\cap\{c\in J| c_1=-1\}$, $f^+_c(x_0)\in C^{\infty}(K)$. Moreover, for any $\epsilon>0$, $m\in \mathbb{N}$, there exists some constant $C(m,K,\epsilon)$ such that
%   \begin{equation}\label{eqcor2_7_1}
%      |\partial^{\alpha}_cf^+_c|\le C(m,K,\epsilon),
%   \end{equation}
%   for any $  0\le |\alpha|\le m$, and $-1<x<1-\epsilon$.
%  % In particular, $\gamma^+(c)$ is $C^{\infty}$ in each $H_i$, $1\le i\le 6$.
%\end{cor}	
\begin{proof}
 %We prove that for any $m\ge 0$ integer, $f^+_c(x_0)\in C^{m}(K)$.
% For any $x_0\in (-1,1)$, there is some $\epsilon>0$ such that $x_0+\epsilon<1$.
 %Given any $\hat{c}\in K$ and $\epsilon>0$ such that $B_{2\epsilon}(\hat{c})\subset\mathring{J}$,
 We know that $U^+_{\theta}(c)$ satisfies (\ref{eq:UthP}) in $(-1,1)$ and $||U^+_{\theta}||_{L^{\infty}(-1,1)}\le C$, where $C$ depends only on $K$.
 By Lemma \ref{lem2_gamma_1}, for any positive integer $m$, there exist some positive constants $\delta$ and $C$, depending only on $m$ and $K$, such that $U_{\theta}^+(c)\in C^m(K \times (-1,-1+\delta))$ and (\ref{eq2_7_00}) holds.% obtained in Lemma \ref{lem:c1c2:LocalSeriesSol} with $U_{\theta}^+(c)(-1)=\tau_2(c_1)$ or $2$. So $U_{\theta}^+(c)(-1+\frac{\delta}{2})\in C^m(K)$. 
 
 %Notice that $f^+_c(x)$ is the solution of the initial value problem
 %Notice the initial value problem
 Consider (\ref{eq_cor2_7}) 
% \begin{equation}\label{eq_cor2_7}
%   \left\{\begin{split}
%      & (1-x^2)U_{\theta}'+2xU_{\theta}+\frac{1}{2}U_{\theta}^2=P_c(x)=c_1(1-x)+c_2(1+x)+c_3(1-x^2),\\% \quad -1<x<1,\\
%      &U_{\theta}(-1+\frac{\delta}{2})=a,  %f^+_c(-1+\frac{\delta}{2}).
%   \end{split}
%   \right.
% \end{equation}
 for $a$ close to $a_0:=U_{\theta}^+(c)(-1+\frac{\delta}{2})$. 
% has a $C^\infty$ solution $U_{\theta}^+(c)$ in $(-1,1)$ if $a=a_0:=U_{\theta}^+(c)(-1+\frac{\delta}{2})$.
 By standard ODE theories, for any $0<\epsilon<2-\delta$, there exist some positive constants $\mu$ and $C$, depending on $m$, $K$ and $\epsilon$, such that if $|a-a_0|<\mu$, then there exists a solution $U_{\theta}\in C^m((a_0-\mu,a_0+\mu)\times K\times [-1+\frac{\delta}{4}, 1-\epsilon])$ of (\ref{eq_cor2_7}), and 
 \[
     |\partial_{a}^{\beta}\partial_c^{\alpha}U_{\theta}|\le C, \quad  |\beta|, |\alpha|\le m, c\in K, -1+\frac{\delta}{4}<x<1-\epsilon.
 \]
 It follows,   also in view of (\ref{eq2_7_00}), that $U^+_{\theta}(c)=U_{\theta}|_{a=a_0}$ satisfies (\ref{eqcor2_7_1}).
 \end{proof}
  %has a $C^1$ solution $U_{\theta}$ in $(-1+\epsilon,1-\epsilon)$. Moreover $U_{\theta}$  depends smoothly on $a$ and $c$ in $(-1+\epsilon,1-\epsilon)$. By the fact that $U_{\theta}^+(c)(-1+\frac{\delta}{2})\in C^m(K)$, we have $U_{\theta}^+(c)\in C^{m}(K\times(-1+\epsilon,1-\epsilon))$ for any $\epsilon>0$. Thus $U_{\theta}^+(c)\in C^{m}(K\times(-1,1))$. % is in $C^{m}(K)$ for all $m\in \mathbb{N}$. 
 
%Moreover,  by Lemma \ref{lem2_gamma_1}, there exists some $\delta>0$ and $C(m,K)$ such that $ |\partial^{\alpha}_cU_{\theta}^+(c)|\le C(m,K)$ for $-1<x<-1+\frac{\delta}{2}$.  Then by standard ODE theory, for any $\epsilon>0$, there exists some $C(m,K,\epsilon)$ such that (\ref{eqcor2_7_1}) is true for all $-1<x<-1+\epsilon$.

 %Similarly we also have $f^-(x_0)$  is in $C^{m}(K)$.

%Let
%	\[
%    H'_1:=\{c\in J | c_1\ge-1,c_2> -1,c_3>\bar{c}_3\},\quad H'_2:=\{c\in J | c_1\ge-1,c_2= -1,c_3>\bar{c}_3\}.
%\]
Similarly to Lemma \ref{lem2_gamma_1} and Corollary \ref{cor2_7} we have
	\addtocounter{lem}{-1}
\renewcommand{\thelem}{\thesection.\arabic{lem}'}%	
\begin{lem}\label{lem2_gamma_1'}
      For any integer $m\ge 0$,  and any compact set $K$ contained in either $J\setminus\{c\in J\mid c_2=-1\}$ or $J\cap\{c\in J\mid c_2=-1\}$, there exist some positive constants $\delta$ and $C$, depending only on $m$ and $K$, such that $U_{\theta}^-(c)\in C^m(K\times (1-\delta, 1))$, and
    \[
        |\partial_c^{\alpha}U_{\theta}^-(c)(x)|\le C, \quad  x\in (1-\delta,1), c\in K,  |\alpha|\le m.
    \]

\end{lem}	
\renewcommand{\thelem}{\thesection.\arabic{lem}}%	

\addtocounter{cor}{-1}
\renewcommand{\thecor}{\thesection.\arabic{cor}'}%	

\begin{cor}\label{cor2_7'}
   For any  $K\subset J\setminus\{c\in J\mid c_2=-1\}$ or $J\cap\{c\in J\mid c_2=-1\}$, $U_{\theta}^-(c)\in C^{\infty}(K\times (-1,1))$. Moreover, for any $\epsilon>0$, $m\in \mathbb{N}$, there exists some positive constant $C$, depending only on $m$, $K$, and $\epsilon$,  such that
   \begin{equation*}%\label{eqcor2_7_1'}
      ||\partial^{\alpha}_cU_{\theta}^-(c)||_{L^{\infty}(-1+\epsilon,1)}\le C, \quad  0\le |\alpha|\le m.
   \end{equation*}
\end{cor}	
%\begin{cor}\label{cor2_7'}
%    For any $x_0\in (-1,1)$, and any $K\subset J\setminus\{c\in J| c_2=-1\}$ or $J\cap\{c\in J| c_2=-1\}$, $f^-_c(x_0)\in C^{\infty}(K)$. Moreover, for any $\epsilon>0$, $m\in \mathbb{N}$, there exists some constant $C(m,K,\epsilon)$ such that
%   \begin{equation}\label{eqcor2_7_1'}
%      |\partial^{\alpha}_cf^-_c|\le C(m,K,\epsilon),
%   \end{equation}
%   for any $  0\le |\alpha|\le m$, and $-1+\epsilon<x<1$.
%   %In particular, $\gamma^-(c)$ is $C^{\infty}$ in each of $H_1'$, $H_2'$, $H_i$, $3\le i\le 6$.
%\end{cor}	
\renewcommand{\thecor}{\thesection.\arabic{cor}}%	

%\begin{rmk}
%It can be actually seen that $H_1=J_1\cup J_3$, $H_2=J_2\cup J_4$, $H_1'=J_1\cup J_2$, $H_2'=J_3\cup J_4$, and $H_{i+2}=J_i\cup J_{i+4}$, $1\le i\le 4$.
%\end{rmk}

Theorem \ref{thm:c1c2} can be obtained from Corollary \ref{cor2_7} and Corollary \ref{cor2_7'}.

%\begin{rmk}
%  By Corollary \ref{cor2_7} and Corollary \ref{cor2_7'}, we know that $\gamma^+(c)$ is $C^{\infty}$ in each $H_i$, $1\le i\le 6$, and $\gamma^-(c)$ is $C^{\infty}$ in each of $H_1'$, $H_2'$, $H_i$, $3\le i\le 6$. But they are not $C^1$ at those points where $c_1=-1$ or $c_2=-1$ in $J$.
%\end{rmk}

%%%%%%%%%%%%%%%%%%%

\bigskip

To prove Theorem \ref{propA_1}, we make the following observations.

%First, when $(c,\gamma)\in F_{k,l}$, $5\le k\le 8$, $f^*_{c_1,c_2}=- \left(\sqrt{1+c_1} + \sqrt{1+c_2} + 2 \right)x+( \sqrt{1+c_1} - \sqrt{1+c_2})$. So $f^*_{c_1,c_2}$ is smooth in $I_{k,l}$.

By Corollary \ref{cor2_7} and Corollary \ref{cor2_7'}, we know that for $1\le k\le 4$ and $l=2,3$, $U_{\theta}^+(c)$ and $U_{\theta}^-(c)$ are smooth in $I_{k,l}$. Here the smoothness means $U_{\theta}^+(c)$ and $U_{\theta}^-(c)$ are smooth restricted to each $I_{k,l}$.

%So to prove Proposition \ref{propA_1}, it remains to prove that $f$ is smooth in $F_{k,3}$ for $1\le k\le 4$ and estimate all $ |\partial^{\alpha}_c\partial_{\gamma}^j f|$.

 By standard ODE theory, since $U_{\theta}$ satisfies (\ref{eq:UthP}), it  is smooth in $I_{k,1}$ for each $1\le k\le 4$.  So  a solution $U_{\theta}$ of the initial value problem
 \begin{equation}\label{eq_prob}
   \left\{\begin{split}
      & (1-x^2)U_{\theta}'+2xU_{\theta}+\frac{1}{2}U_{\theta}^2=P_c(x), \quad -1<x<1,\\ %=c_1(1-x)+c_2(1+x)+c_3(1-x^2)
      & U_{\theta}(0)=\gamma,
   \end{split}
   \right.
 \end{equation}
is smooth with respect to $(c,\gamma)$ in each $I_{k,l}$, $1\le k\le 4$, $1\le l\le 3$. It remains to prove the estimates (i)-(iv) in Theorem \ref{propA_1}. 

%Notice that if $5\le k\le 8$, we estimates can be obtained by the expression of $f^*$ directly. The remaining estimates are obtained from the following estimates.

 We first make some estimates about the solutions $U_{\theta}$ of (\ref{eq_prob}). %Let
%\[
%   I_1=\{(c,\gamma)\mid c_1\ge -1,c_2\ge -1,c_3> \bar{c}_3(c_1,c_2), \gamma^-(c)\le \gamma \le \gamma^+(c)\}
%\]

Recall that for each $(c,\gamma)\in I$, there is  a solution $U_{\theta}=U_{\theta}^{c,\gamma}$ satisfying (\ref{eq_prob}).

\begin{lem}\label{lemAp_1}
  Let $K$ be a compact subset of  $I\setminus\{(c,\gamma)\mid \gamma=\gamma^+(c)\}$. Then for any $\epsilon>0$, there exists some $\delta>0$, depending only on $\epsilon$ and $K$, such that for any $(c,\gamma)\in K$, 
  \[
     |U_{\theta}^{c,\gamma}(x)-U_{\theta}^{c,\gamma}(-1)|<\epsilon, \quad  -1<x<-1+\delta.
  \]
\end{lem}
\begin{proof}
   We prove it by contradiction. Suppose the contrary,  there exist some $\epsilon>0$ and a sequence $(c^i,\gamma^i)\in K$ and $-1<x_i<-1+\frac{1}{i}$, such that 
   \[
      |U_{\theta}^{c^i,\gamma^i}(x_i)-U_{\theta}^{c^i,\gamma^i}(-1)|\ge \epsilon.
   \]
   Since $K$ is compact, there exist a subsequence, still denoted as $(c^i,\gamma^i)$, and some $(c,\gamma)\in K$, such that $(c^i,\gamma^i)\to (c,\gamma)$ as $i\to \infty$.
   
   Denote $U_{\theta}^i=U_{\theta}^{c^i,\gamma^i}$. By standard ODE theory, we have that $U_{\theta}^i\to U_{\theta}:=U_{\theta}^{c,\gamma}$ in $C^1_{loc}(-1,1)$. 
   We first assume that
   \begin{equation}\label{eqAp_00}
     U_{\theta}^i(x_i)\ge U_{\theta}^i(-1)+\epsilon.
   \end{equation}
%since the other case can be treated similarly.
   
   Since $(c,\gamma), (c^i,\gamma^i)\in K$, we have $\gamma<\gamma^+(c)$ and $\gamma^i<\gamma^+(c^i)$. Then, by  Theorem \ref{thm1_3} (ii),  $U_{\theta}^i(-1)=2-2\sqrt{1+c_1^i}$ and $U_{\theta}(-1)=2-2\sqrt{1+c_1}$. 
   
   Since $c^i\to c$, we have $U_{\theta}^i(-1)\to U_{\theta}(-1)$, and therefore for sufficiently large $i$, 
     \begin{equation}\label{eqAp_3}
     U_{\theta}^i(x_i)>U_{\theta}(-1)+\frac{\epsilon}{2}.
   \end{equation}
   
  % Case 1: $K\in I_1\setminus\{c_1=-1\}$.
   \textbf{Case 1}: $U_{\theta}(-1)<2$.
   
    %If $U_{\theta}(-1)<2$, then there exists some fixed $\epsilon_1$,  such that 
    There exists some $\epsilon_1>0$, such that $U_{\theta}(-1)+3\epsilon_1<\min\{2,U_{\theta}(-1)+\frac{\epsilon}{4}\}$. For sufficiently large $i$ we have $U_{\theta}^i(-1)<2-\epsilon_1$. 
    Since $U_{\theta}^i\to U_{\theta}$ in $C^1_{loc}(-1,1)$, we have 
    \[
        \lim_{i\to \infty}U^i_{\theta}(-1+\frac{1}{j})=U_{\theta}(-1+\frac{1}{j}).
    \]
    %Since $U_{\theta}\in C^0[-1,1]$, we have 
    By the continuity of $U_{\theta}$,
    \[
       \lim_{j\to \infty}U_{\theta}(-1+\frac{1}{j})=U_{\theta}(-1).
    \]
      
   % Then there exists some $0<\delta_1<\frac{1}{2}$, such that $U_{\theta}(x)\le U_{\theta}(-1)+\epsilon_1$ for all $-1<x<-1+\delta_1$. Since $U_{\theta}^i\to U_{\theta}$ in $C^1_{loc}(-1,1)$, we have $U_{\theta}^i\to U_{\theta}$ in $C^1[-1+\frac{1}{j}, 1-\frac{1}{j}]$ for each $j$. So for sufficiently large $j$, $U_{\theta}^i(-1+\frac{1}{j})\to U_{\theta}(-1+\frac{1}{j})$. 
    
    %Since $U_{\theta}^i\to U_{\theta}$ in $C^1_{loc}(-1,1)$, we have $U_{\theta}^i\to U_{\theta}$ in $C^1[-1+\delta_j, 1-\delta_j]$. 
    Thus for large $j$, there exists $i_j\ge j$, such that  $-1<x_{i_j}<-1+\frac{1}{j}$ and 
    \[
         U_{\theta}^{i_j}(-1+\frac{1}{j})\le U_{\theta}(-1+\frac{1}{j})+\frac{\epsilon_1}{10}\le U_{\theta}(-1)+2\epsilon_1. 
    \]
    By (\ref{eqAp_3}), $U_{\theta}^{i_j}(x_{i_j})>U_{\theta}(-1)+2\epsilon_1$.
    
    Choose $\tilde{x}_{i_j}\in (x_{i_j},-1+\frac{1}{j})$, satisfying
   \begin{equation*}%\label{eqAp_4}
        U_{\theta}^{i_j}(\tilde{x}_{i_j})= U_{\theta}(-1)+2\epsilon_1\le 2-\frac{\epsilon_1}{2}, \textrm{ and }(U_{\theta}^{i_j})'(\tilde{x}_{i_j})\le 0.
   \end{equation*}
   
   %Notice 
  % \begin{equation}\label{eqAp_6}
  %    (1-\tilde{x}_{i_j}^2)(U_{\theta}^{i_j})'+2\tilde{x}_{i_j}U_{\theta}^i(\tilde{x}_{i_j})+\frac{1}{2}(U_{\theta}^i)^2(\tilde{x}_{i_j})=P_{c^i}(\tilde{x}_{i_j})
 %  \end{equation}
   Plugging $U^{i_j}_{\theta}$ and $\tilde{x}_{i_j}$ in (\ref{eq:UthP}), using the above, we have 
   \begin{equation}\label{eqAp_2}
      2\tilde{x}_{i_j}U_{\theta}^{i_j}(\tilde{x}_{i_j})+\frac{1}{2}(U_{\theta}^{i_j})^2(\tilde{x}_{i_j})\ge P_{c^{i_j}}(\tilde{x}_{i_j}).
   \end{equation}
   Sending $j\to \infty$ in (\ref{eqAp_2}) leads to 
   \[
       h(\xi):=-2\xi+\frac{1}{2}\xi^2\ge P_c(-1),
       \]
       where $\xi:=U_{\theta}(-1)+2\epsilon_1\in (U_{\theta}(-1), 2)$. 
   %Denote $\xi_i=U_{\theta}^{i_j}(\tilde{x}_{i_j})$. %By Theorem \ref{thm1_3}, $f^{c^i,\gamma^-(c^i)}<f_i<f^{c^i,\gamma^+(c^i)}$. Choose some $\hat{c}\in J$ such that $\hat{c}_j\ge c_j$, $j=1,2,3$ for all $(c,\gamma)\in K$. By Lemma \ref{lem:c1c2:compare}, we then have $f^{\hat{c},\gamma^-(\hat{c})}\le f^{c^i,\gamma^-(c^i)}<f_i<f^{c^i,\gamma^+(c^i)}\le f^{\hat{c},\gamma^+(\hat{c})}$ in $(-1,1)$. 
   %By (\ref{eqAp_4}), $\{U_{\theta}^{i_j}(\tilde{x}_{i_j})\}$ is a bounded sequence, so has a convergence subsequence, converging to some $\xi\in (U_{\theta}(-1), 2)$.
   
   Since $h(s)$ is a decreasing function when $s\le 2$, we have
    \[
      h(\xi)< h(U_{\theta}(-1))=P_c(-1),
   \]
   a contradiction.
      
   \textbf{Case 2}: $U_{\theta}(-1)=2$. 
   
  % In this case $f^i(-1)=2-\sqrt{1+c_1^i}\le 2$ and $f^i(-1)\to f(-1)=2$.
    
    By (\ref{eqAp_3}) and the convergence of $U_{\theta}^i(-1)$ to $U_{\theta}(-1)$, we may choose $\tilde{x}_i\in (-1,x_i)$ satisfying
    %Using (\ref{eqAp_3}), the fact that $f_i\to f$ in $C^1_{loc}(-1,1)$ and $f_i(-1)\le 2$, there exist $\tilde{x}_i\to -1$, such that for sufficiently large $i$, 
   \begin{equation*}%\label{eqAp_4_2}
     U_{\theta}^i(\tilde{x}_i)= U_{\theta}(-1)+\frac{\epsilon}{4}=2+\frac{\epsilon}{4}, \textrm{ and }(U_{\theta}^i)'(\tilde{x}_i)\ge 0. 
   \end{equation*}
   Plugging $U^i_{\theta}$ and $\tilde{x}_i$ in (\ref{eq:UthP}), using the above, we have 
   \[
       2\tilde{x}_iU^i_{\theta}(\tilde{x}_i)+\frac{1}{2}(U^i_{\theta}(\tilde{x}_i))^2\le P_{c^i}(\tilde{x}_i).
   \]
   Sending $i\to \infty$, the above leads to 
   \[
      h(2)<h(2+\frac{\epsilon}{4})\le P_c(-1)=h(U_{\theta}(-1))=h(2),
   \]
  a contradiction. 
  
  Now, if instead of (\ref{eqAp_00}), 
  \[
     U^i_{\theta}(x_i)\le U_{\theta}^i(-1)-\epsilon,
  \]
  then for sufficiently large $i$, we have
  \[
      U^i_{\theta}(x_i)< U_{\theta}(-1)-\frac{\epsilon}{2}.
  \]
  As in the proof of Case 1, there exists $\tilde{x}_{i_j}\to -1$, such that 
  \[
     U_{\theta}^{i_j}(\tilde{x}_{i_j})=U_{\theta}(-1)-\frac{\epsilon}{2}=:\xi, \textrm{ and }(U_{\theta}^{i_j})'(\tilde{x}_{i_j})\ge 0.
  \]
  Plugging $U^{i_j}_{\theta}$ and $\tilde{x}_{i_j}$ in (\ref{eq:UthP}), using the above, we have 
   \begin{equation*}%\label{eqAp_5_0}
      2\tilde{x}_{i_j}U_{\theta}^{i_j}(\tilde{x}_{i_j})+\frac{1}{2}(U_{\theta}^{i_j})^2(\tilde{x}_{i_j})\le P_{c^{i_j}}(\tilde{x}_{i_j}).
   \end{equation*}
   Sending $j\to \infty$ in the above leads to 
   \[
       h(\xi):=-2\xi+\frac{1}{2}\xi^2\le P_c(-1).
       \]

   Since $h(s)$ is a decreasing function when $s\le 2$, we have
    \[
      h(\xi)> h(U_{\theta}(-1))=P_c(-1),
   \]
   a contradiction.
\end{proof}

Similarly we have

\addtocounter{lem}{-1}
\renewcommand{\thelem}{\thesection.\arabic{lem}'}%	
	\begin{lem}\label{lemAp_1'}
  Let $K$ be a compact subset of  $I\setminus\{(c,\gamma)\mid \gamma=\gamma^-(c)\}$. Then for any $\epsilon>0$, there exists some $\delta>0$, depending only on $\epsilon$ and $K$, such that for any $(c,\gamma)\in K$, 
  \[
     |U_{\theta}^{c,\gamma}(x)-U_{\theta}^{c,\gamma}(1)|<\epsilon, \quad  1-\delta<x<1.
  \]
\end{lem}
\renewcommand{\thelem}{\thesection.\arabic{lem}}%	

%Let $\alpha_0(c_1)=\sqrt{1+c_1}$.
\begin{lem}\label{lemAp_3}
   Let $K$ be a compact subset of  $I\setminus\{(c,\gamma)\mid c_1=-1 \textrm{ or }\gamma=\gamma^+(c)\}$. Then for any $\epsilon>0$, there exist some positive constants $\delta$ and $C$, depending only on $\epsilon$ and $K$, such that for any $(c,\gamma)\in K$, 
  \[
     |U_{\theta}^{c,\gamma}(x)-U_{\theta}^{c,\gamma}(-1)|\le C(1+x)^{\min\{\sqrt{1+c_1},1\}-\epsilon}, \quad  -1<x<-1+\delta.
  \]
\end{lem}
\begin{proof}
   For convenience, let us denote $U_{\theta}=U_{\theta}^{c,\gamma}$,  $\alpha_0=\sqrt{1+c_1}$, $\tau_1=\tau_1(c_1)=2-2\sqrt{1+c_1}$, and $\tau_2=\tau_2(c_1)=2+2\sqrt{1+c_1}$. Since $\gamma<\gamma^+(c)$, $U_{\theta}^{c,\gamma}(-1)=\tau_1$.
   
   Since $U_{\theta}$ satisfies (\ref{eq_prob}), we have 
   \[
      (1-x^2)(U_{\theta}-\tau_1)'+\frac{U_{\theta}-\tau_2}{2}(U_{\theta}-\tau_1)=h(x):=P_c(x)-P_c(-1)-2(1+x)U_{\theta}.
   \]
   Let $w:=\int_{0}^{x}\frac{U_{\theta}-\tau_2}{2(1-s^2)}ds$. We have 
   \begin{equation}\label{eqAp_3_1}
      U_{\theta}-\tau_1=(\gamma-\tau_1) e^{-w}+e^{-w}\int_{0}^{x}e^{w}\frac{h}{1-s^2}ds.
   \end{equation}
   
   By Lemma \ref{lemAp_1}, there exists some $\delta=\delta(\epsilon,K)>0$, such that $|U_{\theta}(x)-\tau_1|<\epsilon$ for all $x\in (-1,-1+\delta)$. By Lemma \ref{lem2_0_0}, $|U_{\theta}(x)|\le C(K)$ and therefore $|h(x)|\le C(K)(1+x)$ for $x\in (-1,1)$. So for all $x\in (-1,-1+\delta)$, 
   \[
       w<\int_{-1+\delta}^{x}\frac{\tau_1-\tau_2-\epsilon}{2(1-s^2)}ds+C(K)\le (-\alpha_0-\frac{\epsilon}{4})\ln (1+x)+C(\epsilon,K),
   \]
   and 
   \[
      w>\int_{-1+\delta}^{x}\frac{\tau_1-\tau_2+\epsilon}{2(1-s^2)}ds-C(K)\ge (-\alpha_0+\frac{\epsilon}{4})\ln (1+x)-C(\epsilon,K).
   \]
   Thus
   \[
      e^{w}\le C(\epsilon,K)(1+x)^{-\alpha_0-\frac{\epsilon}{4}}, \quad e^{-w}\le C(\epsilon, K)(1+x)^{\alpha_0-\frac{\epsilon}{4}}.
   \]
   Plugging this into (\ref{eqAp_3_1}), we have 
   \[
      |U_{\theta}-\tau_1|\le C(\epsilon,K)(1+x)^{\alpha_0-\frac{\epsilon}{4}}+C(\epsilon,K)(1+x), \quad -1<x<-1+\delta.
   \]
\end{proof}

%Similarly, let $\alpha_0'(c_2)=\sqrt{1+c_2}$,  we have
\addtocounter{lem}{-1}
\renewcommand{\thelem}{\thesection.\arabic{lem}'}%	
	\begin{lem}\label{lemAp_3'}
   Let $K$ be a compact subset of  $I\setminus\{(c,\gamma)\mid c_2=-1 \textrm{ or }\gamma=\gamma^-(c)\}$. Then for any $\epsilon>0$, there exists some positive constants $\delta$ and $C$, depending only on $\epsilon$ and $K$, such that for any $(c,\gamma)\in K$, 
  \[
     |U_{\theta}^{c,\gamma}(x)-U_{\theta}^{c,\gamma}(1)|\le C(1-x)^{\min\{\sqrt{1+c_2},1\}-\epsilon}, \quad  1-\delta<x<1.
  \]
\end{lem}
\renewcommand{\thelem}{\thesection.\arabic{lem}}%	

\bigskip

\begin{lem}\label{lemAp_2}
Let $K$ be a compact subset of  $I\cap \{(c,\gamma)\mid c_1=-1, \gamma<\gamma^+(c)\}$. Then for any $\epsilon>0$, there exists some $\delta>0$, depending only on $\epsilon$ and $K$, such that for any $(c,\gamma)\in K$, 
  \[
     |(U_{\theta}^{c,\gamma}-2)\ln \left(\frac{1+x}{3}\right)-4|<\epsilon, \quad -1<x<-1+\delta.
  \]
\end{lem}
\begin{proof}
If $U_{\theta}:=U_{\theta}^{c,\gamma}$ is a solution of  (\ref{eq_prob}) with $(c,\gamma)\in I$,  $c_1=-1$, and $\gamma<\gamma^+(c)$,  we have $U_{\theta}(-1)=2$. Denote 
\begin{equation*}
     g:=g^{c,\gamma}=(U_{\theta}-2)\ln  \left(\frac{1+x}{3}\right), \quad -1<x<0.
     \end{equation*}
     
Then by Theorem 1.3 in \cite{LLY1}, $g(-1)=4$, $g(x)$ satisfies 

\begin{equation}\label{eqAp_2_1}
(1-x^2)\ln \left(\frac{1+x}{3}\right) g'-(1-x)g+\frac{1}{2}g^2=H_{c,\gamma}(x):=(P_c(x)-2(1+x)U_{\theta}+2)\left(\ln \frac{1+x}{3}\right)^2.
%(1-x^2)\ln\left(\frac{1+x}{3}\right)g'-(1-x)g+\frac{1}{2}g^2=H_{c,\gamma}(x):=(P_c(x)-2(1+x)f+2)\left(\ln\frac{1+x}{3}\right)^2.
\end{equation}
%and $g(0)=(2-\gamma)\ln 3$.

   We prove the lemma by contradiction. Assume there exist some $\epsilon>0$ and a sequence $(c^i,\gamma^i)\in K$ and $-1<x_i<-1+\frac{1}{i}$, such that 
   \[
      |g^{c^i,\gamma^i}(x_i)-g^{c^i,\gamma^i}(-1)|= |g^{c^i,\gamma^i}(x_i)-4|\ge \epsilon.
   \]
   Since $K$ is compact, there exist a subsequence, still denoted as $(c^i,\gamma^i)$, and some $(c,\gamma)\in K$, such that $(c^i,\gamma^i)\to (c,\gamma)$ as $i\to \infty$.
   
   Denote $g_i=g^{c^i,\gamma^i}$. By standard ODE theory, we have that $g_i\to g:=g^{c,\gamma}$ in $C^1_{loc}(-1,1)$. As explained earlier, $g(-1)=4$.
   
  % Without loss of generality, we assume that 
  We first assume that
   \begin{equation}\label{eqAp_7_00}
      g_i(x_i)\ge 4+\epsilon.
   \end{equation}
   Using this and the fact that $g_i\to g$ in $C^1_{loc}(-1,1)$, by similar arguments as in the proof of Lemma \ref{lemAp_1}, we have that there exist $x_i\le \tilde{x}_i\to -1$, such that 
   \begin{equation*}%\label{eqAp_7}
      \xi_i:=g_i(\tilde{x}_i)= 4+\frac{\epsilon}{4}=:\xi, \textrm{ and }g'_i(\tilde{x}_i)\le 0.
   \end{equation*}
      
   Let $h(s):=-2s+\frac{1}{2}s^2$. By (\ref{eqAp_2_1}) we have that 
   \[
       -(1-\tilde{x}_i)g_i(\tilde{x}_i)+\frac{1}{2}g^2_i(\tilde{x}_i)\le H_{c^i,\gamma^i}(\tilde{x}_i).
   \]
   Sending $i\to\infty$, we have 
   \[
       h(\xi)\le H_{c,\gamma}(-1)=0.
   \]
   On the other hand, since $\xi >4$, so $h(\xi)>0$. A contradiction.
   
   Now if instead of (\ref{eqAp_7_00}), we have
   \[
       g_i(x_i)\le 4-\epsilon.
   \]
   Without loss of generality, we assume that $0<\epsilon<1$. As in Case 1  of the proof of Lemma \ref{lemAp_1}, there exists $\tilde{x}_{i_j}\to -1$, such that 
   \[
      g_{i_j}(\tilde{x}_{i_j})=4-\frac{\epsilon}{2}=:\xi, \textrm{ and }g_{i_j}'(\tilde{x}_{i_j})\ge 0.
   \]
   By (\ref{eqAp_2_1}) we have that 
   \[
       -(1-\tilde{x}_{i_j})g_{i_j}(\tilde{x}_{i_j})+\frac{1}{2}g^2_{i_j}(\tilde{x}_{i_j})\ge H_{c^{i_j},\gamma^{i_j}}(\tilde{x}_{i_j}).
   \]
   Sending $i\to\infty$, we have 
   \[
       h(\xi)\ge H_{c,\gamma}(-1)=0.
   \]
   On the other hand, since $3<\xi <4$, so $h(\xi)<h(4)=0$. A contradiction.
   \end{proof}

Similarly,   we have 
\addtocounter{lem}{-1}
\renewcommand{\thelem}{\thesection.\arabic{lem}'}%	
	\begin{lem}\label{lemAp_2'}
Let $K$ be a compact subset of  $I\cap \{ (c,\gamma)\mid c_2=-1, \gamma>\gamma^-(c)\}$. Then for any $\epsilon>0$, there exists some $\delta>0$, depending only on $\epsilon$ and $K$, such that for any $(c,\gamma)\in K$, 
  \[
     |(U_{\theta}^{c,\gamma}+2)\ln \left(\frac{1-x}{3}\right)+4|<\epsilon, \quad  1-\delta<x<1.
  \]
\end{lem}
\renewcommand{\thelem}{\thesection.\arabic{lem}}%

The next lemma strengthens Lemma \ref{lemAp_2}.
\begin{lem}\label{lemAp_4}
Let $K$ be a compact subset of  $I\cap \{(c,\gamma)\mid c_1=-1, \gamma<\gamma^+(c)\}$. Then for any $\epsilon>0$, there exists some positive constants $\delta$ and $C$, depending only on $\epsilon$ and $K$, such that for any $(c,\gamma)\in K$, 
  \[
     \left|U_{\theta}^{c,\gamma}(x)-2-\frac{4}{\ln \frac{1+x}{3}}\right|<C\left|\ln \frac{1+x}{3}\right|^{-2+\epsilon}, \quad  -1<x<-1+\delta.
  \]
\end{lem}

\begin{proof}
 For convenience let us denote $U_{\theta}=U_{\theta}^{c,\gamma}$. Let $V:=U_{\theta}-2-\frac{4}{\ln\frac{1+x}{3}}$. Then $V$ satisfies the equation
 \[
     (1-x^2)V'+\frac{4}{\ln \frac{1+x}{3}}V+\frac{1}{2}V^2 =h(x),
 \]
 where $h:=P_c(x)-P_c(-1)-\frac{4(1+x)}{\left(\ln \frac{1+x}{3}\right)^2}-2(1+x)V-4(1+x)-\frac{8 (1+x)}{\ln\frac{1+x}{3}}$. We have, using Lemma \ref{lemAp_1}, that there exists some $\delta=\delta(\epsilon,K)$, such that $|h|\le C(\epsilon,K)(1+x)$ for all $x\in (-1,-1+\delta)$.
 
 Let $w:=\int_{-\frac{1}{2}}^{x}\frac{\frac{1}{2}V+\frac{4}{\ln\frac{1+s}{3}}}{1-s^2}ds$. We have 
 \begin{equation}\label{eqAp_4_1}
   V=V\left(-\frac{1}{2}\right)e^{-w}+e^{-w}\int_{-\frac{1}{2}}^{x}e^w\frac{h}{1-s^2}ds.
 \end{equation}
 Since $V\left(-\frac{1}{2}\right)=U_{\theta}\left(-\frac{1}{2}\right)-2+\frac{4}{\ln 6}$, we have $|V\left(-\frac{1}{2}\right)|\le C(K)$. By Lemma \ref{lemAp_2}, making  $\delta=\delta(\epsilon,K)>0$ smaller if necessary, we have  $|(U_{\theta}-2)\ln\frac{1+x}{3}-4|<\epsilon$, i.e. $|V|\le \frac{\epsilon}{\left|\ln\frac{1+x}{3}\right|}$, for all $-1<x<-1+\delta$. We also have $|h|\le C(\epsilon,K)(1+x)$ for all $x\in (-1,-1+\delta)$. Thus for all $-1<x<-1+\delta<-\frac{1}{2}$, we have 
 \[
     w\le C(\epsilon,K)+ \int_{-1+\delta}^{x}\frac{4+\frac{\epsilon}{2}}{(1-s^2)\ln\frac{1+s}{3}}\le C(\epsilon,K)+(2+\epsilon)\ln(-\ln\frac{1+x}{3}),
 \]
 and 
 \[
     w\ge -C(\epsilon,K)+\int_{-1+\delta}^{x}\frac{4-\frac{\epsilon}{2}}{(1-s^2)\ln\frac{1+s}{3}}\ge -C(\epsilon,K)+(2-\epsilon)\ln(-\ln\frac{1+x}{3}).
 \]
 So 
 \[
    e^{w}\le C(\epsilon,K)\left|\ln\frac{1+x}{3}\right|^{2+\epsilon}, \quad e^{-w}\le C(\epsilon,K)\left|\ln\frac{1+x}{3}\right|^{-2+\epsilon}.
 \]
 Plugging this into (\ref{eqAp_4_1}), we have 
 \[
    |V|\le C(\epsilon,K)\left|\ln\frac{1+x}{3}\right|^{-2+\epsilon}.%+C(\epsilon,K)(1+x).
 \]
 The proof is finished.
\end{proof}

Similarly we have the following strengthening of Lemma \ref{lemAp_2'}.
\addtocounter{lem}{-1}
\renewcommand{\thelem}{\thesection.\arabic{lem}'}%	
\begin{lem}\label{lemAp_4'}
Let $K$ be a compact subset of  $I\cap \{(c,\gamma)\mid c_2=-1, \gamma>\gamma^-(c)\}$. Then for any $\epsilon>0$, there exists some positive constants $\delta$ and $C$, depending only on $\epsilon$ and $K$, such that for any $(c,\gamma)\in K$, 
  \[
     |U_{\theta}^{c,\gamma}(x)+2+\frac{4}{\ln\frac{1-x}{3}}|<C\left|\ln\frac{1-x}{3}\right|^{-2+\epsilon}, \quad  1-\delta<x<1.
  \]
\end{lem}
\renewcommand{\thelem}{\thesection.\arabic{lem}}%

%We have the following estimates

\bigskip

\bigskip

Now using Lemma \ref{lemAp_1}--Lemma \ref{lemAp_4'}, we prove the following estimates of partial derivatives of $U_{\theta}:=U_{\theta}^{c,\gamma}$ with respect to $(c,\gamma)$ on each $I_{k,l}$.

\begin{lem}\label{lemB_3_1}
  For any $\epsilon>0$, $m\in \mathbb{N}$,  and compact subset $K$ of $ I \setminus \{(c,\gamma)\mid c_1=-1\textrm{ or }\gamma=\gamma^+(c)\}$, %I_{1,3}, I_{1,2}, I_{3,2}$  and $I_{3,3}$,
   there exists some positive constant $C$, depending only on $m$, $K$, and $\epsilon$, such that
   \begin{equation*}%\label{eqcorB_1}
      \sum_{1\le |\alpha|+j\le m}|\partial^{\alpha}_c\partial_{\gamma}^j U_{\theta}|\le C,\quad -1<x<1-\epsilon.
   \end{equation*}
  % for any $  1\le |\alpha|+j\le m$, and $-1<x<1-\epsilon$.
\end{lem}
\begin{proof}
 We prove the lemma by induction. We use $C(m,K,\epsilon)$ and $C$ to denote constants which may be different from line to line,  and their dependence is clear from the context.
 
	We know by (\ref{eq_prob}) that 
	$$
		(1-x^2) \left( \frac{\pt U_{\theta}}{\pt \gamma}\right) ' + \left(2x + U_{\theta} \right) \left( \frac{\pt U_{\theta}}{\pt \gamma}\right) = 0, 
	$$
	$$
		(1-x^2) \left( \frac{\pt U_{\theta}}{\pt c_i}\right) ' + \left(2x + U_{\theta}\right) \left( \frac{\pt U_{\theta}}{\pt c_i}\right) = \partial_{c_i}P_c(x),
	$$
	and $\frac{\pt U_{\theta}(0)}{\pt \gamma}=1$, $\frac{\pt U_{\theta}(0)}{\pt c_i}=0$, $i=1,2,3$.
Denote 
	\begin{equation}\label{eqB_3_1_1}
		a(x) = a_{c,\gamma}(x) = \int_0^x \frac{2s+U_{\theta}}{1-s^2} ds. 
	\end{equation}
	Then
\begin{equation}\label{eqB_3_2}
		 \frac{\pt U_{\theta}}{\pt \gamma} =  e^{-a(x)},
		 \end{equation}
		 and for $i=1,2,3$,
\begin{equation}\label{eqB_3_3}		 
		 \frac{\pt U_{\theta}}{\pt c_i} =  e^{-a(x)}  \int_0^x e^{a(s)} \frac{\partial_{c_i}P_c(s)}{1-s^2} ds.
	\end{equation}
	 
By the definition of $a(x)$, Lemma \ref{lem2_0_0} and Lemma \ref{lemAp_3}, we have that there exists some constant $C=C(\epsilon, K)$ such that
	$$
		e^{-a(x)} \le C(1+x)^{1-\frac{U_{\theta}(-1)}{2}}, \quad e^{a(x)} \le C (1+x)^{\frac{U_{\theta}(-1)}{2}-1},\quad -1<x<1-\epsilon. 
	$$
	Since when $(c,\gamma)\in K$, $U_{\theta}(-1)<2$, there exists some $C(K,\epsilon)$, such that $e^{-a(x)}\le C(K,\epsilon)$. Thus by (\ref{eqB_3_2}) and (\ref{eqB_3_3}) we have that for $-1<x<1-\epsilon$,
	\begin{equation*}%\label{eqB_3_4}
		\sum_{|\alpha|+j=1}|\pt_{c}^{\alpha}\pt_\gamma^j U_{\theta}| \le C(K,\epsilon).
			\end{equation*}
% for all $|\alpha|+j=1$.

 Now for $m\ge 2$, suppose that $C(m_1,K,\epsilon)$ exist for all $1\le m_1\le m-1$, then for any $|\alpha|+j=m$,
 \[
   (1-x^2)(\partial_c^{\alpha}\partial_{\gamma}^jU_{\theta})'+2x\partial_c^{\alpha}\partial_{\gamma}^jU_{\theta}+\frac{1}{2}\partial_c^{\alpha}\partial_{\gamma}^j(U_{\theta}^2)=\partial_c^{\alpha}\partial_{\gamma}^jP_c(x).
 \]
This leads to
 \[
     (1-x^2)(\partial_c^{\alpha}\partial_{\gamma}^jU_{\theta})'+(2x+U_{\theta})\partial_c^{\alpha}\partial_{\gamma}^jU_{\theta}=h,
     \]
     where 
     \[
        h:=-\frac{1}{2}\sum_{0\le (\alpha_1, j_1)\le (\alpha,j),0<|\alpha_1|+j_1<m}\left( \begin{matrix}
		\alpha\\
		\alpha_1
	\end{matrix} \right)\left( \begin{matrix}
		j\\
		j_1
	\end{matrix} \right)\partial_c^{\alpha_1}\partial_{\gamma}^{j_1}U_{\theta}\partial_c^{\alpha-\alpha_1}\partial_{\gamma}^{j-j_1}U_{\theta}.
 \]
 Notice that $\partial_c^{\alpha}\partial_{\gamma}^jU_{\theta}(0)=0$ for all $|\alpha|+j\ge 2$, we have
 \[
    \partial_c^{\alpha}\partial_{\gamma}^jU_{\theta}=%Ce^{-a(x)}+
    e^{-a(x)}\int_{0}^{x}e^{a(s)}\frac{h(s)}{1-s^2}ds.
 \]

 By the induction assumption, $h\in L^{\infty}(-1,1-\epsilon)$ and there exists some positive constant $C$, depending only on $m$, $K$, and $\epsilon$ such that $|h|_{L^{\infty}(-1,1-\epsilon)}\le C$.  
 So we have 
 \[
      |\partial_c^{\alpha}\partial_{\gamma}^jU_{\theta}|_{L^{\infty}(-1,1-\epsilon)}\le C.
 \]
The proof is finished.
\end{proof}

Similarly, using Lemma \ref{lem2_0_0} and Lemma \ref{lemAp_3'} we have 

\addtocounter{lem}{-1}
\renewcommand{\thelem}{\thesection.\arabic{lem}'}%	
\begin{lem}\label{lemB_3_1'}
    For any $\epsilon>0$, $m\in \mathbb{N}$,  and compact subset $K$ of $I\setminus\{ (c,\gamma)\mid c_2=-1 \textrm{ or }\gamma=\gamma^-(c)\}$,      there exists some positive constant $C$, depending only on $m$, $K$, and $\epsilon$, such that
   \begin{equation*}%\label{eqcorA_1'}
     \sum_{1\le |\alpha|+j\le m} |\partial^{\alpha}_c\partial_{\gamma}^jU_{\theta}|\le C, \quad -1+\epsilon<x<1.
   \end{equation*}
  % for any $  1\le |\alpha|+j\le m$, and $-1+\epsilon<x<1$.
\end{lem}	
\renewcommand{\thelem}{\thesection.\arabic{lem}}%	

\begin{lem}\label{lemB_3_2}
  For any $\epsilon>0$, $m\in \mathbb{N}$,  and compact subset $K$ of $I\cap \{(c,\gamma)\mid c_1=-1, \gamma<\gamma^+(c)\}$,%$\in I_{2,3},  I_{2,2}, I_{4,3}$ or $I_{4,2}$, 
  there exists some positive constant $C$, depending only on $m$, $K$, and $\epsilon$, such that 
   \begin{equation*}%\label{eqB_3_2_1}
      \sum_{1\le |\alpha|+j\le m, \alpha_1=0}\left(\ln \frac{1+x}{3}\right)^2|\partial^{\alpha}_c\partial_{\gamma}^jU_{\theta}|\le C, \quad -1<x<1-\epsilon.
   \end{equation*}
   %for any $\alpha=(0,\alpha_2,\alpha_3)$ and $j\ge 0$ satisfying $1\le |\alpha|+j\le m$, and $-1<x<1-\epsilon$.
\end{lem}
\begin{proof}
We prove the lemma by induction. Denote $C(m,K,\epsilon)$ and $C$ to be constants which may vary from line to line, and their dependence is clear from the context. Similar as the proof of Lemma \ref{lemB_3_1}, we have (\ref{eqB_3_2}) and (\ref{eqB_3_3}) where $a(x)$ is defined by (\ref{eqB_3_1_1}). %Notice that in this case
 %By Lemma \ref{lemAp_4},  there exists some constant $C=C(m,K)$
%\begin{equation}\label{eqB_3_2_2}
%		f = 2 + \frac{4}{\ln(1+x)} + \frac{O(1)}{(\ln (1+x))^2},\quad -1<x<1-\epsilon.
%	\end{equation}
By the definition of $a(x)$, Lemma \ref{lem2_0_0} and Lemma \ref{lemAp_4},  there exists some constant $C=C(m,K,\epsilon)$, such that
	$$
		e^{-a(x)} \le C\left(\ln\frac{1+x}{3}\right)^{-2}, \quad \quad e^{a(x)} \le C\left(\ln\frac{1+x}{3}\right)^{2}. 
	$$
	Notice in this case, $i=2$ or $3$ in (\ref{eqB_3_3}), and $|\partial_{c_i}P_c|\le C(1+x)$ for some constant $C$ depending only on $K$, so we have that for $-1<x<1-\epsilon$,
	\begin{equation*}%\label{eqB_3_2_3}
		\sum_{|\alpha|+j=1} \left(\ln \frac{1+x}{3}\right)^2|\pt_{c}^{\alpha}\pt_\gamma^j U_{\theta}| \le C(K,\epsilon).
			\end{equation*}
% for all $|\alpha|+j=1$.
 
 Now suppose that $C(m_1,K,\epsilon)$ exists for all $1\le m_1\le m-1$. As in the proof of the previous lemma we have, for all $ |\alpha|+j=m$ and $\alpha_1=0$, that
  \[
    \partial_c^{\alpha}\partial_{\gamma}^j U_{\theta}=Ce^{-a(x)}+e^{-a(x)}\int_{-1+\frac{\delta}{2}}^{x}e^{a(s)}\frac{h(s)}{1-s^2}ds,
 \] 
 where 
  \[
     h:=-\frac{1}{2}\sum_{0\le (\alpha_1,j_1)\le (\alpha, j), 0<|\alpha_1|+j_1<m}\left( \begin{matrix}
		\alpha\\
		\alpha_1
	\end{matrix} \right)\left( \begin{matrix}
		j\\
		j_1
	\end{matrix} \right)\partial_c^{\alpha_1}\partial_{\gamma}^{j_1}U_{\theta}\partial_c^{\alpha-\alpha_1}\partial_{\gamma}^{j-j_1}U_{\theta}.
 \]
 Then, by the induction assumption, $h\in L^{\infty}(-1,1-\epsilon)$ and there is some positive constant $C$, depending only on $m$, $K$, and $\epsilon$, such that $ \left(\ln \frac{1+x}{3}\right)^4|h(x)|\le C$ for all $-1<x<1-\epsilon$. Using this estimate we then have 
 \[
    \left(\ln \frac{1+x}{3}\right)^2|\pt_{c}^{\alpha}\pt_\gamma^j U_{\theta}| \le C.
 \]
 
 The lemma is proved.
\end{proof}

Similarly, using Lemma \ref{lemAp_4'} we have 
\addtocounter{lem}{-1}
\renewcommand{\thelem}{\thesection.\arabic{lem}'}%	
\begin{lem}\label{lemB_3_2'}
    For any $\epsilon>0$, $m\in \mathbb{N}$,  and compact subset $K$ of $I\cap \{c,\gamma)\mid c_2=-1, \gamma>\gamma^-(c)\}$, %$\in I_{3,3}, I_{4,3},I_{3,1}$ or $I_{4,1}$, 
    there exists some positive constant $C$, depending only on $m$, $K$, and $\epsilon$, such that
   \begin{equation*}%\label{eqB_3_2_1'}
      \sum_{1\le |\alpha|+j\le m, \alpha_2=0}\left(\ln \frac{1-x}{3}\right)^2|\partial^{\alpha}_c\partial_{\gamma}^jU_{\theta}|\le C(m,K,\epsilon),\quad -1+\epsilon<x<1.
   \end{equation*}
   %for any $\alpha=(0,\alpha_2,\alpha_3)$ and $j\ge 0$ satisfying $1\le |\alpha|+j\le m$, and $-1+\epsilon<x<1$
\end{lem}	
\renewcommand{\thelem}{\thesection.\arabic{lem}}%	

Theorem \ref{propA_1} follows from Corollary \ref{cor2_7}, Corollary \ref{cor2_7'}, Lemma \ref{lemB_3_1}, \ref{lemB_3_1'}, \ref{lemB_3_2} and \ref{lemB_3_2'}.

\bigskip

\end{document}